\documentclass[10pt,a4paper]{article}
\usepackage{authblk}
\usepackage[utf8]{inputenc}
\usepackage[T1]{fontenc}
\usepackage{amsmath}
\usepackage{amssymb}
\usepackage{graphicx}

\usepackage{subcaption}
\def\sizethreebyone{0.72}
\def\sizetwobyone{0.48}
\def\sizeonebyone{0.24}

\usepackage[margin=2.5cm]{geometry}
\usepackage{multirow}

\usepackage{enumerate}
\usepackage{pdfpages}
\usepackage{mathtools}
\usepackage{dsfont}
\usepackage{physics}
\usepackage{siunitx}
\mathtoolsset{showonlyrefs}
\usepackage{algorithm}
\usepackage{algpseudocode}

\usepackage{amsthm}
\usepackage{thmtools}
\newtheoremstyle{break}{\topsep}{\topsep}{\itshape}{}{\bfseries}{}{\newline}{}
\theoremstyle{break}

\newtheorem{proposition}{Proposition}[section]
\newtheorem{lemma}{Lemma}[section]
\newtheorem{remark}{Remark}[section]

\usepackage[backend=biber, style=numeric,giveninits, doi=true, url=false, eprint=true, isbn=false]{biblatex}
\AtEveryBibitem{\clearfield{issn, publisher}}
\AtEveryCitekey{\clearfield{issn, publisher}}
\usepackage[colorlinks=true, allcolors=blue, hypertexnames=false]{hyperref}

\title{A numerical Fourier cosine expansion method with higher order Taylor schemes for fully coupled FBSDEs}
\author[1,2]{Balint Negyesi\thanks{Corresponding author. Email: \href{mailto:b.negyesi@proton.me}{b.negyesi@proton.me}.}} 
\author[2]{Cornelis W. Oosterlee}
\affil[1]{\small Delft Institute of Applied Mathematics (DIAM), Delft University of Technology}
\affil[2]{\small Mathematical Institute, Utrecht University}

\begin{document}
	\maketitle

    \begin{abstract}
        A higher-order numerical method is presented for scalar valued, coupled forward-backward stochastic differential equations. Unlike most classical references, the forward component is not only discretized by an Euler-Maruyama approximation but also by higher-order Taylor schemes. This includes the famous Milstein scheme, providing an improved strong convergence rate of order 1; and the simplified order 2.0 weak Taylor scheme exhibiting weak convergence rate of order 2. In order to have a fully-implementable scheme in case of these higher-order Taylor approximations, which involve the derivatives of the decoupling fields, we use the COS method built on Fourier cosine expansions to approximate the conditional expectations arising from the numerical approximation of the backward component. Even though higher-order numerical approximations for the backward equation are deeply studied in the literature, to the best of our understanding, the present numerical scheme is the first which achieves strong convergence of order 1 for the whole coupled system, including the forward equation, which is often the main interest in applications such as stochastic control. Numerical experiments demonstrate the proclaimed higher-order convergence, both in case of strong and weak convergence rates, for various equations ranging from decoupled to the fully-coupled settings.
    \end{abstract}
	
	\tableofcontents
	
		\section{Introduction}
	In this paper we are concerned with solving the following fully-coupled system of forward-backward stochastic differential equations (FBSDE)\begin{subequations}\label{eq:fbsde}\noeqref{eq:fbsde:sde, eq:fbsde:bsde}
        \begin{align}
		X_t &= x_0 + \int_0^t \mu(s, X_s, Y_s, Z_s)\mathrm{d}s + \int_0^t \sigma(s, X_s, Y_s, Z_s)\mathrm{d}W_s,\label{eq:fbsde:sde}\\
		Y_t &= g(X_T) + \int_t^T f(s, X_s, Y_s, Z_s)\mathrm{d}s - \int_t^T Z_s\mathrm{d}W_s,\label{eq:fbsde:bsde}
	\end{align}
    \end{subequations}
	where all coefficients are deterministic, scalar valued functions, $T>0$ is a finite time horizon, and $\{W_t\}_{t\in[0, T]}$ is a standard Brownian motion over an appropriate probability space.
    Such equations have an innate connection with second-order, parabolic, quasi-linear PDEs with terminal boundaries of the following form
	\begin{align}\label{eq:pde}
        \begin{split}
            \partial_t u(t, x) + \frac{1}{2}\widetilde{\sigma}^2(t, x, u, \partial_x u)\partial_{xx}^2 u(t, x) + \mu(t, x, u, \partial_x u \sigma) \partial_x u(t, x) +f(t, x, u, \partial_x u\sigma)&=0,\\
		u(T, x) &= g(x),
        \end{split}
	\end{align}
    where $\widetilde{\sigma}$ is connected to $\sigma$ and the coefficients $\mu, f, g$ coincide with those in \eqref{eq:fbsde} -- see e.g. \cite{zhang_backward_2017, fromm_theory_2015}. In particular, due to non-linear generalizations of the Feynman-Kac formula, the solution of \eqref{eq:pde} is related to the solution triple of \eqref{eq:fbsde} by the following relations
    \begin{align}\label{eq:feynman-kac}
        Y_t=u(t, X_t),\quad Z_t=\partial_x u(t, X_t)\sigma(t, X_t, u(t, X_t))\eqqcolon v(s, X_s),
    \end{align}
    at least in the case when the diffusion coefficient does not depend on $Z$
    \cite{pardoux_backward_1992, zhang_backward_2017}.

    Fully-coupled FBSDEs naturally arise in applications of stochastic control, \cite{yong_stochastic_1999, pham_continuous-time_2009}. Solving such equations analytically is seldom possible in an analytical fashion and one most often has to resort to numerical approximations instead. The literature in the so called \emph{decoupled} framework, i.e. when the coefficients $\mu, \sigma$ do not depend on $Y$ and $Z$ are thoroughly studied, see e.g. \cite{bouchard_discrete-time_2004}. However, things get more complicated in the \emph{coupled} setting, i.e. when the coefficients $\mu, \sigma$ in \eqref{eq:fbsde:sde} take $Y, Z$ as arguments. Due to the nature of the coupling, the forward simulation of \eqref{eq:fbsde:sde} is not straightforward, as one needs to have approximated the backward equation's solution pair \eqref{eq:fbsde:bsde} in order to do so beforehand.    
    Several approaches have been proposed in the past decades, starting with the famous \emph{four step scheme} \cite{ma_solving_1994}, built on the connection with the associated quasi-linear PDEs \eqref{eq:feynman-kac}. Without the sake of completeness, we mention \cite{bender_time_2008, milstein_numerical_2006, cvitanic_steepest_2005, ma_forward-backward_2007, zhang_backward_2017} and the references therein.

    Classical techniques include Monte Carlo techniques, branching diffusions, quantization algorithms \cite{delarue_forward_2006}, Fourier cosine expansions. In particular, starting with the decoupled setting Ruijter and Oosterlee \cite{ruijter_fourier_2015} proposed an algorithm for the numerical resolution of decoupled FBSDEs, built on the COS method \cite{fang_novel_2009}, where the corresponding conditional expectations are approximated by means of Fourier cosine expansions enabled by the fact that, at least in the Markovian framework, the \emph{characteristic function} of the transition between two time steps is known in closed form, given a suitable discretization of \eqref{eq:fbsde:sde}.
    In their original work, the forward diffusion was assumed to be a standard Arithmetic Brownian motion, giving pathwise analytical solutions for \eqref{eq:fbsde:sde}, eliminating the need for discrete time approximations. Later on, they extended this framework in \cite{ruijter_numerical_2016} to more general, state and time dependent drift and diffusion coefficients, up to discretizations by second-order Taylor schemes, including the \emph{Euler-Maruyama} and \emph{Milstein} schemes. In case of the former, their algorithm exhibited a strong convergence rate of order $1/2$, whereas in case of the latter, this was improved to order $1$. Nevertheless, these extensions were given in the decoupled setting, i.e. $Y, Z$ did not enter the drift and diffusion coefficients in \eqref{eq:fbsde:sde}.
    
    In order to remedy this,  following their ideas, Huijskens et al. in \cite{huijskens_efficient_2016} extended their algorithm to the coupled FBSDE framework of \eqref{eq:fbsde}. Therein, the forward and backward equations are first discretized by order $1/2$ Euler schemes, and are subsequently decoupled by suitable \emph{decoupling fields} which approximate the deterministic mappings in \eqref{eq:feynman-kac}. Thereafter, they proposed three different strategies for finding the \textit{optimal} decoupling fields: the so called \emph{explicit method}, in which throughout the backward recursion, the decoupling relations at time $t_n$ are replaced by the already computed approximations at time step $t_{n+1}$; a \emph{local method} where the decoupling relations are obtained via Picard iterations at each time step; and a \emph{global method} where Picard iterations take place simultaneously, at all points in time similar to \cite{bender_time_2008}. They found that the explicit method, inspired by \cite{delarue_forward_2006}, is the most robust, giving strong convergence rate of order $1/2$, even in the case of fully-coupled systems, where the diffusion coefficient $\sigma$ depends also on $Z$, as in \eqref{eq:fbsde:sde}.

    The purpose of the present paper is to combine \cite{ruijter_numerical_2016} with \cite{huijskens_efficient_2016} for the coupled system \eqref{eq:fbsde}. Doing so, we extend the COS approximations of \cite{huijskens_efficient_2016} to higher-order discretizations for fully-coupled FBSDE systems, such that strong convergence rates of order $1$ and weak convergence rates of order $2$ can be achieved. In particular, we generalize the discretization of the forward diffusion in \eqref{eq:fbsde:sde} to second-order Taylor schemes , see e.g. \cite{kloeden_numerical_1992}, including the \emph{Milstein}- and \emph{simplified order 2.0 weak Taylor} schemes. Subsequently, in order to preserve the higher-order convergence rates in the forward diffusion, we solve the corresponding backward recursions in the approximations of the backward equation by appropriate second-order schemes \cite{zhao_generalized_2012, crisan_second_2014}, in order to obtain a suitable pair of decoupling fields \eqref{eq:feynman-kac}. In this sense, our work is a generalization of \cite{ruijter_numerical_2016}, extending the approaches proposed therein from the decoupled to the fully-coupled setting.
    The main challenge in doing so, is that higher-order corrections terms involve \emph{derivatives} of the decoupling fields in \eqref{eq:feynman-kac}. In order to remedy this, we capitalize on the fact that COS approximations are infinitely differentiable, and given an \emph{explicit} scheme, these derivatives can be computed directly at each step in the backward recursion. This enables higher-order discretization schemes, up to strong convergence rates of order $1$ with a \emph{Milstein}-, and weak convergence rates of order $2$ with a \emph{simplified order 2.0 weak Taylor} scheme, when the characteristic function of the corresponding Markov transition is still available in closed form.

    The rest of the paper is organized as follows. Section \ref{sec:discretization} is devoted to the discrete time approximations of \eqref{eq:fbsde:sde}. After introducing key notations and theoretical concepts, general, second-order schemes are formulated for the forward component \eqref{eq:fbsde:sde}, followed by the standard, recursive sequence of conditional expectations related to the solution pair of the backward equation \eqref{eq:fbsde:bsde}. In order to make these discrete time approximations fully-implementable, section \ref{sec:bcos} explains the BCOS method in the coupled setting. Our main results are given in section \ref{sec:numerical}, where numerical experiments are presented for a selection of coupled FBSDE systems, exhibiting the proclaimed rate of strong and weak convergence rates for a wide range of problems.

    \section{Discrete time approximations of forward-backward stochastic differential equations}\label{sec:discretization}
    Throughout the paper we are working on a filtered probability space $(\Omega, \mathcal{F}=\cap_{t\in [0, T]} \mathcal{F}_t, \mathbb{P})$, where $\mathcal{F}_t$ is the natural filtration generated by the Brownian motion $\{W_t\}_{t\in[0, T]}$ augmented with the usual $\mathbb{P}$-null sets. 
    The solution of \eqref{eq:fbsde} is a triple of predictable stochastic processes $\{(X_t, Y_t, Z_t)\}_{t\in[0, T]}$ such that \eqref{eq:fbsde} is satisfied $\mathbb{P}$ almost surely, and the processes satisfy the following integrability conditions
    \begin{align}
        \mathbb{E}\left[\sup_{t\in [0, T]}\abs{X_t}^2\right]+ \mathbb{E}\left[\sup_{t\in [0, T]}\abs{Y_t}^2\right]+\mathbb{E}\left[\int_0^T \abs{Z_s}^2\mathrm{d}s\right]<\infty.
    \end{align}
    The well-posedness of \eqref{eq:fbsde} has been deeply studied, and established under by now classical assumptions, see e.g. \cite{zhang_backward_2017, pardoux_backward_1992}. For the rest of paper we consider equations such that any of such conditions is satisfied, and assume that \eqref{eq:fbsde} has a unique solution triple $\{(X_t, Y_t, Z_t)\}_{t\in[0, T]}$.

    In what follows, the following standard notations are used. Conditional expectations at $t_n\in[0, T]$ are denoted by $\mathbb{E}_n[\cdot]\coloneqq \mathbb{E}[\cdot\vert \mathcal{F}_{t_{n}}]$. For a function $f: [0, T]\times \mathbb{R}\times \mathbb{R}\times\mathbb{R}\to \mathbb{R}$, we put $\partial_t f$ for the partial derivative in time, and $\partial_x f, \partial_y f, \partial_z f$ for the corresponding spatial derivatives. Second-order derivatives are denoted by $\partial_{yx}^2 f=\partial_y(\partial_x f)$, and analogously for all other partial derivatives. Given a function $f: (t, x, y, z)\mapsto f$, and two deterministic mappings $\varphi: x\mapsto \varphi$, $\zeta: x\mapsto \zeta$, we put $\bar{f}^{\varphi, \zeta}(t, x)\coloneqq f(t, x, \varphi(x), \zeta(x))$, consequently
    \begin{align}\label{eq:total_derivative_in_x}
        \partial_x \bar{f}^{\varphi, \zeta}(t, x)\equiv \partial_x f(t, x, \varphi(x), \zeta(x)) + \partial_y f(t, x, \varphi(x), \zeta(x))\partial_x \varphi(x)+\partial_z f(t, x, \varphi(x), \zeta(x))\partial_x \zeta(x).
    \end{align}
    For the rest of the paper, $\pi\coloneqq \{0=t_0<t_1<\dots<t_N=T\}$, denotes a partitioning of $[0, T]$, with $\Delta t_n\coloneqq t_{n+1}-t_{n}$, and $|\pi|\coloneqq \max_{n=0, \dots, N-1} \Delta t_n$. The Brownian increments between two adjacent time points are given by $\Delta W_n\coloneqq W_{t_{n+1}} - W_{t_{n}}$. For an equidistant time grid we use $h\coloneqq T/N$, $t_n=nh$, $n=0, \dots, N-1$. We will first discuss the discrete time approximations of the system \eqref{eq:fbsde} that are investigated in this paper.

    \subsection{Forward discretization}
    The literature on numerical approximations for forward stochastic differential equations is vast, for a classical reference we refer to \cite{kloeden_numerical_1992}. However, in case of coupled FBSDEs as in \eqref{eq:fbsde:sde} numerical approximations become more intricate due to the coefficient functions $\mu$ and $\sigma$ depending on not just the state process $X$, but also the solution pair of the backward equation $(Y, Z)$ in \eqref{eq:fbsde:bsde}. In order to enable the numerical approximation of forward equation, one first needs to \emph{decouple} \eqref{eq:fbsde:sde} from \eqref{eq:fbsde:bsde}. Motivated by the Feynman-Kac formulae in \eqref{eq:feynman-kac}, the standard technique is to construct a so called \emph{decoupling field}, which is a deterministic mapping of $X$ at each point in time. In fact, the Markov property of the solution pair of \eqref{eq:fbsde:bsde} together with \eqref{eq:feynman-kac} imply the existence of two \emph{deterministic} mappings of time and state, such that the solution to the following forward SDE
    \begin{align}\label{eq:sde:decoupled}
        X_t^{u, v} = x_0 + \int_{0}^t \mu(s, X_s^{u, v}, u(s, X_s^{u, v}), v(s, X_s^{u, v}))\mathrm{d}s + \int_0^t \sigma(s, X_s^{u, v}, u(s, X_s^{u, v}), v(s, X_s^{u, v}))\mathrm{d}W_s
    \end{align}
    coincides with that of \eqref{eq:fbsde:sde} pathwise, almost surely. Therefore, given the true decoupling fields $u, v$, one can construct a \emph{decoupled} FBSDE system consisting of \eqref{eq:sde:decoupled} and \eqref{eq:fbsde:bsde}, and proceed with standard approximations techniques developed for the decoupled framework in order to numerically solve the entire system in \eqref{eq:fbsde}.
    The main difficulty in the approach described above is that the decoupling relations $u, v$ are not known in advance, one has to provide numerical approximations of them in the discrete time framework, combined with the numerical approximation of the backward equation \eqref{eq:fbsde:bsde}.
    
    To this end, let us fix two adjacent points $t_n, t_{n+1}$ in the time partition and a decoupling pair $\varphi, \zeta: \mathbb{R}\to\mathbb{R}$.
    Following standard techniques of It\^{o}-Taylor expansions, see e.g. \cite{ruijter_numerical_2016} and the references therein, one can approximate the solution $\{X_t^{u, v}\}_{t\in (t_n, t_{n+1}]}$ of \eqref{eq:sde:decoupled} by the second-order Taylor expansions of the general form
    \begin{align}\label{eq:sde:general_second_order_scheme}
        X_t^{t_n, x, \varphi, \zeta} = x + \bar{m}^{\varphi, \zeta}(t, x)(t-t_n) + \bar{s}^{\varphi, \zeta}(t_n, x)(W_{t}-W_{t_{n}}) + \bar{\kappa}^{\varphi, \zeta}(t_n, x)(W_{t}-W_{t_{n}})^2,\quad X_{t_{n}}^{u, v}=x,
    \end{align}
    for $t\in(t_n, t_{n+1}]$, given $\varphi(x), \zeta(x)$ are accurate approximations of $u(t_n, x)$ and $v(t_n, x)$, respectively, for any $x$. We use the notation $X_{t_{m}}^{t_n, x, \varphi, \zeta}=X_m^{\pi, \varphi, \zeta}$ for $m=n, n+1$. The general second-order approximation in \eqref{eq:sde:general_second_order_scheme} includes famous discretization schemes such as the Euler-Maruyama scheme defined by
    \begin{align}\label{eq:coefficients:euler}
        \bar{m}^{\varphi, \zeta}(t, x)= \bar{\mu}^{\varphi, \zeta}(t, x),\quad \bar{s}^{\varphi, \zeta}(t, x)=\bar{\sigma}^{\varphi, \zeta}(t, x),\quad \bar{\kappa}^{\varphi, \zeta}(t, x)=0;
    \end{align}
    the Milstein scheme defined by
    \begin{align}\label{eq:coefficients:milstein}
        \bar{m}^{\varphi, \zeta}(t, x)= \bar{\mu}^{\varphi, \zeta}(t, x) - \bar{\kappa}^{\varphi, \zeta}(t, x),\quad \bar{s}^{\varphi, \zeta}(t, x)=\bar{\sigma}^{\varphi, \zeta}(t, x),\quad \bar{\kappa}^{\varphi, \zeta}(t, x)=\bar{\sigma}^{\varphi, \zeta}(t, x)\partial_x \bar{\sigma}^{\varphi, \zeta}(t, x)/2;
    \end{align}
    or the \emph{simplified order 2.0 weak Taylor} scheme given by
    \begin{align}\label{eq:coefficients:2.0-weak-taylor}
            \begin{split}
                \bar{m}^{\varphi, \zeta}(t, x)&=\begin{aligned}[t]
                    &\bar{\mu}^{\varphi, \zeta}(t, x)-\bar{\sigma}^{\varphi, \zeta}(t, x)\partial_x \bar{\sigma}^{\varphi, \zeta}(t, x)/2\\
                    &+ \big(\partial_t \bar{\mu}^{\varphi, \zeta}(t, x)+\bar{\mu}^{\varphi, \zeta}(t, x)\partial_x\bar{\mu}^{\varphi, \zeta}(t, x)+\partial_{xx}^2\bar{\mu}^{\varphi, \zeta}(t, x)(\bar{\sigma}^{\varphi, \zeta}(t, x))^2/2\big)\Delta t_n/2
                \end{aligned}\\
                \bar{s}^{\varphi, \zeta}(t, x)&=\begin{aligned}[t]
                    &\bar{\sigma}^{\varphi, \zeta}(t, x)\\
                    &\begin{aligned}[t]
                        + \big(\partial_x \bar{\mu}^{\varphi, \zeta}(t, x)\bar{\sigma}^{\varphi, \zeta}(t, x)+\partial_t \bar{\sigma}^{\varphi, \zeta}(t, x)+\bar{\mu}^{\varphi, \zeta}(t, x)\partial_x \bar{\sigma}^{\varphi, \zeta}(t, x)&\\
                        +\partial_{xx}^2\bar{\sigma}^{\varphi, \zeta}(t, x)(\bar{\sigma}^{\varphi, \zeta}(t, x))^2/2&\big)\Delta t_n/2
                    \end{aligned}
                \end{aligned}\\
                \bar{\kappa}^{\varphi, \zeta}(t, x)&=\bar{\sigma}^{\varphi, \zeta}(t, x)\partial_x \bar{\sigma}^{\varphi, \zeta}(t, x)/2.
            \end{split}
    \end{align}
    It is well-known, see e.g. \cite{kloeden_numerical_1992}, that under standard assumptions the classical discretization schemes above converge to the continuous solution of \eqref{eq:sde:decoupled} and thus also that of \eqref{eq:fbsde:sde}, with strong and weak convergence rates as in table \ref{tab:convergence-rates},
    \begin{table}[t]
        \centering
        \begin{tabular}{l|ccc}
             &  Euler-Maruyama & Milstein & simplified order 2.0 weak Taylor\\
             \hline
            weak & $1$ & $1$ & $2$\\
            strong & $1/2$ & $1$ & $1$
        \end{tabular}
        \caption{Weak and strong convergence rates of \eqref{eq:sde:general_second_order_scheme} given different discretization schemes}
        \label{tab:convergence-rates}
    \end{table}
    where the convergence rates are defined as follows
    \begin{align}
        \sup_{t\in [0, T]}(\mathbb{E}[|X_t^{u, v} - X_t^{\pi, \varphi, \zeta}|^2])^{1/2}\leq C\abs{\pi}^{\gamma_s}, \quad |\mathbb{E}[p(X_t^{u, v})] - \mathbb{E}[p(X_t^{\pi, \varphi, \zeta})]|\leq C\abs{\pi}^{\gamma_w}.
    \end{align}
    In above $p(x)$ is any $2(\gamma_w+1)$ continuously differentiable function of polynomial growth, and $C$ is a generic constant independent of $\abs{\pi}$. Strong convergence implies convergence in probability to the true solution of \eqref{eq:sde:decoupled}, whereas weak convergence provides relevant information about the solution at $t=0$ whenever the initial condition of the forward SDE is deterministic, which in many financial applications is of  special relevance.

    \begin{remark}
        Many higher-order Taylor expansions similar to \eqref{eq:sde:general_second_order_scheme} could be considered such as the \emph{order 1.5 strong Taylor} -- see e.g. \cite{kloeden_numerical_1992} --, including more terms of the corresponding It\^{o}-Taylor expansion, involving powers of the Brownian increment $W_{t}-W_{t_{n}}$ which are higher than $2$. The main reason why we restrict our further analysis to second-order schemes of the form \eqref{eq:sde:general_second_order_scheme} is due to the COS approximations that follow. In fact, when the highest power of the Brownian increment in \eqref{eq:sde:general_second_order_scheme} is at most $2$, the corresponding Markov transition's characteristic function can be computed in closed form -- see lemma \ref{lemma:chf} below -- which is a key component of the COS method applied hereafter. However, for schemes of order higher than $2$ this property no longer holds, and in order to be able to compute the COS formula, one would first have to numerically approximate the corresponding characteristic function.
    \end{remark}

    \subsection{Backward discretizations}
    Let us turn to the discrete time approximations of the backward equation \eqref{eq:fbsde:bsde}. In light of \eqref{eq:sde:decoupled}, we fix a decoupling field $(\varphi, \zeta)$ and consider the corresponding decoupled BSDE defined by
    \begin{align}\label{eq:bsde:decoupled}
        Y_t^{\varphi, \zeta} = g(X_T^{\varphi, \zeta}) + 
        \int_t^T f(s, X_s^{\varphi, \zeta}, Y_s^{\varphi, \zeta}, Z_s^{\varphi, \zeta})\mathrm{d}s - \int_t^T Z_s^{\varphi, \zeta}\mathrm{d}W_s,
    \end{align}
    which coincides with the solution pair of \eqref{eq:fbsde:bsde} when $\varphi=u, \zeta=v$, established by \eqref{eq:feynman-kac}.
    Given that the objective is to formulate a higher-order scheme in the fully-coupled setting, e.g. without knowing an exact decoupling pair $(u, v)$, it is fundamental that the corresponding discrete time approximations of \eqref{eq:bsde:decoupled} admit a higher-order convergence rate than that of the classical backward Euler scheme of \cite{bouchard_discrete-time_2004}. In what follows we consider two second-order schemes from the BSDE literature.
    Zhao et al. in \cite{zhao_generalized_2012} proposed a generalized theta-scheme which approximates the Markovian solution pair of \eqref{eq:bsde:decoupled} by the following backward recursion of conditional expectations
    \begin{align}\label{eq:zhao_scheme}
    \begin{split}
        y(t_N, x)&=g(x),\quad z(t_N, x)=\partial_x g(x)\sigma(t_N, y(t_N, x), z(t_N, x)),\\
        z(t_n, x) &= \frac{1}{\theta_3 \Delta t_n}\begin{aligned}[t]
            \big(&\theta_4 \Delta t_n \mathbb{E}_n^x [z(t_{n+1}, X_{n+1}^{\pi, \varphi, \zeta})] + (\theta_3-\theta_4)\mathbb{E}_n^x[y(t_{n+1}, X_{n+1}^{\pi, \varphi, \zeta})\Delta W_n]\\
            &+(1-\theta_2)\Delta t_n \mathbb{E}_n^x[f(t_{n+1}, y(t_{n+1}, X_{n+1}^{\pi, \varphi, \zeta}), z(t_{n+1}, X_{n+1}^{\pi, \varphi, \zeta}))\Delta W_n]\big),
        \end{aligned}\\
        y(t_n, x)&=\begin{aligned}[t]
            &\mathbb{E}_n^x[y(t_{n+1}, X_{n+1}^{\pi, \varphi, \zeta}) + (1-\theta_1)\Delta t_n f(t_{n+1}, y(t_{n+1}, X_{n+1}^{\pi, \varphi, \zeta}), z(t_{n+1}, X_{n+1}^{\pi, \varphi, \zeta}))]\\
            &+ \theta_1\Delta t_n f(t_{n}, y(t_{n}, x), z(t_{n}, x)),
        \end{aligned}
    \end{split}
    \end{align}
    with $\theta_1, \theta_2\in [0, 1]$, $\theta_3\in(0, 1]$, $\abs{\theta_4}\leq \theta_3$ -- see also their related works in  \cite{zhao_new_2006, zhao_error_2009}.
    We remark that the generalized theta-scheme in \eqref{eq:zhao_scheme} includes many classical discretization schemes such as the (implicit) backward Euler scheme of Bouchard and Touzi \cite{bouchard_discrete-time_2004} with $\theta_1=\theta_2=\theta_3=1, \theta_4=0$; or the theta-scheme considered in \cite{ruijter_fourier_2015, huijskens_efficient_2016} with $\theta_2=\theta_3, \theta_4=\theta_3-1$ and also the one of \cite{ruijter_numerical_2016} with the extra condition $\theta_1=\theta_2$. In \cite{zhao_generalized_2012} the authors show that the generalized theta-scheme in \eqref{eq:zhao_scheme} has a strong convergence rate of order $2$ in the decoupled framework, when $\theta_1=\theta_2=\theta_3=1/2$ and $\theta_4\leq \theta_3$, given that the underlying forward diffusion is a Brownian motion, i.e. $\mu(t, x, y, z)=0, \sigma(t, x, y, z)=1$ in \eqref{eq:sde:decoupled}, and the coefficients of the BSDE $g, f$ are sufficiently smooth with bounded derivatives. This makes the scheme \eqref{eq:zhao_scheme} a suitable choice for the second-order approximation of \eqref{eq:fbsde}, as using the estimates in \eqref{eq:zhao_scheme} for the decoupling of \eqref{eq:sde:general_second_order_scheme} induces errors that scale with order of at most 2, preserving the convergence rates of a second-order Taylor scheme.

    Alternatively, Crisan and Manolarakis in \cite{crisan_second_2014} proposed a second-order discretization for decoupled FBSDEs, defined by the following backward recursion
    \begin{align}\label{eq:crisan_scheme}
    \begin{aligned}[t]
        y(t_N, x)&=g(x),\quad z(t_N, x)=\partial_x g(x),\\
        z(t_n, x) &= \begin{aligned}[t]
            &\frac{4\Delta t_n + 6t_n}{(\Delta t_n)^2}\mathbb{E}_n^x[(y(t_{n+1}, x) + \Delta t_n f(X_{n+1}^{\pi, \varphi, \zeta}, y(t_{n+1}, X_{n+1}^{\pi, \varphi, \zeta}), z(t_{n+1}, X_{n+1}^{\pi, \varphi, \zeta}))\Delta W_n]\\
            &-\frac{6}{(\Delta t_n)^2}\mathbb{E}_n^x\big[\big(\int_{t_{n}}^{t_{n+1}} s\mathrm{d}W_s\big)(y(t_{n+1}, x) + \Delta t_n f(X_{n+1}^{\pi, \varphi, \zeta}, y(t_{n+1}, X_{n+1}^{\pi, \varphi, \zeta}), z(t_{n+1}, X_{n+1}^{\pi, \varphi, \zeta}))\big],
        \end{aligned}\\
        y(t_n, x) &= \begin{aligned}[t]
            &\mathbb{E}_n^x[y(t_{n+1}, X_{n+1}^{\pi, \varphi, \zeta}) + \Delta t_n/2 f(X_{n+1}^{\pi, \varphi, \zeta}, y(t_{n+1}, X_{n+1}^{\pi, \varphi, \zeta}), z(t_{n+1}, X_{n+1}^{\pi, \varphi, \zeta})] \\&+ \Delta t_n/2 f(x, y(t_{n}, x), z(t_{n}, x).
        \end{aligned}
    \end{aligned}
    \end{align}
    In \cite{crisan_second_2014}, the authors prove second order convergence in $|\pi|$ for \eqref{eq:crisan_scheme} under sufficiently smooth coefficients, generalizing second-order convergence beyond Brownian noise. Furthermore, they show that whenever those assumptions are not satisfied, their scheme still preserves the strong convergence rate of the backward Euler scheme of order $1/2$.

    We note that the discretization of $Y$ is always implicit in \eqref{eq:crisan_scheme} and also in \eqref{eq:zhao_scheme} whenever $\theta_1>0$.

    \paragraph{Single versus multi-step backward schemes.}
    It is important to mention that higher-order discretizations of BSDEs have been thoroughly studied, and many important results have been established in this regard. In fact, besides the two one-step schemes mentioned above, many multi-step approaches have been developed over the past decades, which may guarantee higher order convergence rates. Without the sake of completeness, we mention \cite{chassagneux_linear_2014, chassagneux_rungekutta_2014, zhao_stable_2010, zhao_numerical_2014} in the decoupled setting, and \cite{zhao_new_2014, teng_high-order_2021} for the coupled framework. In case of the latter two, one can constuct a multi-step discrete time approximation scheme for the BSDE \eqref{eq:fbsde:bsde} which reads as follows
    \begin{align}
        y(t_N, x)&=g(x),\quad z(t_N, x)=\partial_x g(x)\sigma(t_N, x),\\
        z(t_n, x)&=\sum_{j=1}^{k}\alpha_{j}^k\mathbb{E}_n^x[(W_{t_{n+j}} - W_{t_{n}})y(t_{n+j}, X_{n+j}^\pi)]\\
        y(t_n, x)&=\frac{1}{\alpha_0^k}\big(-\sum_{j=1}^k \alpha_j^k \mathbb{E}_{n}^x[y(t_{n+j}, X_{n+j}^\pi)] - f(t_n, x, y(t_n, x), z(t_n, x))\big),
    \end{align}
    where the $\{\alpha_j^k\}_{j=0, \dots, k}$ are known explicitly for every $k=1, \dots, 6$. The scheme is supposed to converge with a rate of $\mathcal{O}(\abs{\pi}^k)$, at least in the weak sense. The main difference between the thereby proposed multi-step approaches and that of the present paper is twofold.
    First, the $k$-step multi-step schemes are not immediately implementable without appropriate approximations for the first $Y, Z$ at the first $k$ time steps closest to $T$. In other words, one has to compute $y(t_{N-j}, x), z(t_{N-j}, x)$ for $j=1, \dots, k$ on either a finer time grid/using a higher-order FBSDE method, such as e.g. \cite{zhao_generalized_2012, crisan_second_2014}, in order to be able to keep the same rate of convergence. Second, all multi-step schemes mentioned above require the numerical approximation of $k$ conditional expectations with transitions between time step $t_n$ and $t_{n+j}, j=1, \dots, k$. As we shall see in the next section, in the context of the COS method this would imply the computation of the transition matrix $\Phi_n$ in \eqref{def:Phi}, $k$ times for each time step, which is the computationally most expensive part of the algorithm proposed therein. Therefore, a one-step scheme such as \eqref{eq:zhao_scheme}-\eqref{eq:crisan_scheme} is computationally preferable compared to multi-step alternatives when the spatial approximations are given by the COS method.
    Finally, and most importantly, the multi-step approximations for coupled FBSDEs in \cite{zhao_new_2014, zhao_numerical_2014, teng_high-order_2021} are only providing higher-order convergence of the backward equation's solution pair $Y, Z$ in \eqref{eq:fbsde:bsde}, and not that of \eqref{eq:fbsde:sde} which is still discretized by an Euler-Maruyama scheme. In other words, the somewhat surprising conclusion of \cite{zhao_new_2014, teng_high-order_2021} is that higher-order convergence of $Y, Z$ can be achieved even with a lower order scheme for the forward diffusion. In some applications, such as option pricing, this is sufficient as the main interest is the solution of the backward equation. However, in many other applications, e.g. stochastic optimal control, the forward diffusion is the main quantity of interest, in which case one wants to provide a higher-order scheme for the whole system in \eqref{eq:fbsde}, making higher-order approximations for the forward diffusion inevitable.

    For the reasons above, we discretize the backward equation by the one-step schemes \eqref{eq:zhao_scheme}-\eqref{eq:crisan_scheme}. Both discrete time approximation scheme are only implementable given a machinery which approximates the conditional expectations on the left hand side. In our case, this will be done by the COS method, explained in the upcoming section.

	\section{COS approximations}\label{sec:bcos}

    The COS method originally proposed in \cite{fang_novel_2009} is a Fourier based method to approximate conditional expectations, given that the underlying randomness is generated by a Markov transition whose characteristic function is known in (semi-)analytical closed form. 
    In what follows we fix $t_n$, and the corresponding decoupling fields $\varphi, \zeta$ in \eqref{eq:sde:decoupled}. Then, we are interested in the Markov transition $X_n^{\pi, \varphi, \zeta}=x\mapsto X_{n+1}^{\pi, \varphi, \zeta}$ defined by the following second-order Taylor scheme
    \begin{align}\label{eq:sde:discretization:decoupled}
		X_{n+1}^{\pi, \varphi, \zeta} = x + \bar{m}^{\varphi, \zeta}(t_n, x)\Delta t_n + \bar{s}^{\varphi, \zeta}(t_n, x)\Delta W_n + \bar{\kappa}^{\varphi, \zeta}(t_n, x)(\Delta W_n)^2.
	\end{align}
    The COS approximation of a conditional expectation of some function of $X_{n+1}^{\pi, \varphi, \zeta}$ then reads as follows
    \begin{align}\label{eq:cos_approximation:type1}
        \mathbb{E}_n^x[v(t_{n+1}, X_{n+1}^{\pi, \varphi, \zeta})]\approx \sideset{}{'}\sum_{k=0}^{K-1} \mathcal{V}_k(t_{n+1})\Re{\phi_{{X}_{n+1}^{\pi, \varphi, \zeta}}(k\pi/(b-a)|t_n, x)\exp(-ik\pi a/(b-a))},
    \end{align}
    where $\phi_{{X}_{n+1}^{\pi, \varphi, \zeta}}$ is the \emph{characteristic function} of $X_{n+1}^{\pi, \varphi, \zeta}$ and the \emph{Fourier cosine expansion coefficients} are defined by
    \begin{align}\label{eq:def:cosine_expansion_coefficients}
        \mathcal{V}_k(t_{n+1})&\coloneqq \frac{2}{b-a}\int_{a}^b v(t_{n+1}, x)\cos(\frac{k\pi}{b-a}(x-a))\mathrm{d}x,\\ v(t_{n+1}, x)&=\sideset{}{'}\sum_{k=0}^{\infty} \mathcal{V}_k(t_{n+1})\cos(\frac{k\pi}{b-a}(x-a)),
    \end{align}
    and the notation $\sum'$ means that the $0$'th term in the summation is halved.
 
    For any given triple $\bar{m}^{\varphi, \zeta}, \bar{s}^{\varphi, \zeta}, \bar{\kappa}^{\varphi, \zeta}$, the following lemma, see \cite[lemma 3.1]{ruijter_numerical_2016}, establishes an explicit, closed form expression for the characteristic function of the Markov transition $X_n^{\pi, \varphi, \zeta}=x\to X_{n+1}^{\pi, \varphi, \zeta}$ given a decoupling field $(\varphi, \zeta)$.
	\begin{lemma}[Characteristic function of Markov transitions, \cite{ruijter_numerical_2016}]\label{lemma:chf}
		Consider \eqref{eq:sde:discretization:decoupled}, for any decoupling pair $\varphi, \zeta$, the characteristic function of $X_{n+1}^{\pi, \varphi, \zeta}$ given $X_n^{\pi, \varphi, \zeta}=x$ reads as follows		\begin{align}\label{eq:sde:characteristic_function}
			\phi^{\pi, \varphi, \zeta}_{n}(u\vert X_n^{\pi, \varphi, \zeta}=x)\begin{aligned}[t]
				&\coloneqq\mathds{E}_n^x\left[\exp(iuX_{n+1}^{\pi, \varphi, \zeta})\vert X_{n}^{\pi, \varphi, \zeta}=x\right]\\
				&=\frac{\exp(iu(x+\bar{m}^{\varphi, \zeta}(t_n, x)\Delta t_n) - \frac{u^2(\bar{s}^{\varphi, \zeta}(t_n, x))^2\Delta t_n}{2(1-2iu\bar{\kappa}^{\varphi, \zeta}(t_n, x)\Delta t_n)})}{\sqrt{(1-2iu\bar{\kappa}^{\varphi, \zeta}(t_n, x)\Delta t_n)}}.
			\end{aligned} 
		\end{align}
	\end{lemma}
    \begin{proof}
        The proof is analogous to that of \cite[lemma 3.1]{ruijter_numerical_2016}.
        Due to $\Delta W_n\sim \mathcal{N}(0, \Delta t_n)$, for an $x$ such that $\bar{\kappa}^{\varphi, \zeta}(t_n, x)=0$, the transition in \eqref{eq:sde:discretization:decoupled} is normal with mean $x + \bar{\mu}^{\varphi, \zeta}(t_n, x)\Delta t_n$ and variance $(\bar{s}^{\varphi, \zeta}(t_n, x))^2\Delta t_n$. Equation \eqref{eq:sde:characteristic_function} is then found by substituting into the characteristic function of a normal distribution.\\
        For an $x$ such that $\bar{\kappa}^{\varphi, \zeta}(t_n, x)\neq 0$, one can write
        \begin{align}\label{eq:chf:proof:step}
            X_{n+1}^{\pi, \varphi, \zeta} \begin{aligned}[t]
                &=\begin{aligned}[t]
                x + \bar{\mu}^{\varphi, \zeta}(t_n, x)\Delta t_n&-(\bar{s}^{\varphi, \zeta}(t_n, x))^2/(4\bar{\kappa}^{\varphi, \zeta}(t_n, x)) \\
                &+ \bar{\kappa}^{\varphi, \zeta}(t_n, x)\big(\Delta W_n + \bar{s}^{\varphi, \zeta}(t_n, x)/(2\bar{\kappa}^{\varphi, \zeta}(t_n, x))\big)^2
            \end{aligned}\\
                &\overset{\textnormal{d}}{=} x + \bar{\mu}^{\varphi, \zeta}(t_n, x)\Delta t_n-(\bar{s}^{\varphi, \zeta}(t_n, x))^2/(4\bar{\kappa}^{\varphi, \zeta}(t_n, x)) + \bar{\kappa}^{\varphi, \zeta}(t_n, x)\Delta t_n \mathcal{X},
            \end{aligned}
        \end{align}
        where $\mathcal{X}\sim\chi'^{2}_\nu(\lambda(x))$ follows a noncentral $\chi^2$ distribution with degrees of freedom $\nu=1$ and non-centrality parameter 
        $\lambda=(\bar{s}^{\varphi, \zeta}(t_n, x))^2/(2\bar{\kappa}^{\varphi, \zeta}(t_n, x)\sqrt{\Delta t_n})^2$, whose characteristic function is given by
        \begin{align}\label{eq:chf:noncentral_chi_squared}
            \phi_{\chi'^{2}_{\nu(\lambda)}}(u|x)=\exp(i\lambda u/(1-2iu))/(1-2iu)^{-\nu/2}.
        \end{align}
        Combining \eqref{eq:chf:noncentral_chi_squared} with \eqref{eq:chf:proof:step}, we finally obtain \eqref{eq:sde:characteristic_function}.
    \end{proof}
    The BCOS method, originating from \cite{ruijter_fourier_2013} and later extended in \cite{ruijter_numerical_2016, huijskens_efficient_2016} proposes to approximate the conditional expectations in the discrete time approximations schemes of BSDEs collected in \eqref{eq:zhao_scheme} and \eqref{eq:crisan_scheme} by the COS formula \eqref{eq:cos_approximation:type1}. Indeed, by virtue of lemma \ref{lemma:chf}, the conditional expectations defining the approximations of $Y$ at time step $n$ are fully-implementable, given the availability of Fourier cosine expansion coefficients at time step $t_{n+1}$. However, in order to be able to apply the COS method on the aforementioned discretizations, one also needs to solve conditional expectations of the form
    \begin{align}
        \mathbb{E}_n^x\big[v(t_{n+1}, X_{n+1}^{\pi, \varphi, \zeta})(\Delta W_n)^k\big],\qquad k\in\mathbb{N},
    \end{align}
    which appear in the approximations for $Z$ in \eqref{eq:zhao_scheme} and \eqref{eq:crisan_scheme}.
    Given Fourier cosine expansion coefficients for the deterministic function $y\mapsto v(t_{n+1}, y)$ defined by \eqref{eq:def:cosine_expansion_coefficients}, one has
	\begin{align}
		\mathds{E}_n^x[v(t_{n+1}, X_{n+1}^{\pi, \varphi, \zeta})(\Delta W_n)^k] \begin{aligned}[t]
			&=\sideset{}{'}\sum_{l=0}^{\infty} \mathcal{V}_l(t_{n+1})\mathds{E}_n^x\left[\cos(l\pi\frac{X_{n+1}^{\pi, \varphi, \zeta} - a}{b-a})(\Delta W_n)^k\right]\\
			&=\sideset{}{'}\sum_{l=0}^{\infty} \mathcal{V}_l(t_{n+1})\Re{\mathds{E}_n^x\left[\exp(il\pi\frac{X_{n+1}^{\pi, \varphi, \zeta}}{b-a})(\Delta W_n)^k\right]\exp(-il\pi a/(b-a))}\\
			&=\sideset{}{'}\sum_{l=0}^{\infty} \mathcal{V}_l(t_{n+1}) \Re{J_k(x\vert l\pi/(b-a))\exp(-il\pi a/(b-a))},
		\end{aligned}
	\end{align}
    where we put
	\begin{align}\label{eq:def:j_k}
		J_k(x\vert u)\coloneqq \mathds{E}_n^x\left[\exp(iuX_{n+1}^{\pi, \varphi, \zeta})(\Delta W_n)^k\right],\qquad k\in\mathbb{N}.
	\end{align}
    Given the Markov transition \eqref{eq:sde:discretization:decoupled}, the following lemma generalizes the truncated series expansion argument in \cite[eq. 3.31]{ruijter_numerical_2016}, and is established by an integration by parts argument.
	\begin{lemma}[Integration by parts formulae with discretization \eqref{eq:sde:discretization:decoupled}]\label{lemma:j_k}
		Given the discretizations \eqref{eq:sde:discretization:decoupled}, the conditional expectations of the form \eqref{eq:def:j_k} admit for each $k\geq 1$
		\begin{align}
			J_k(x\vert u) = \frac{iu\bar{s}^{\varphi, \zeta}(t_n, x)\Delta t_n}{1-2iu\bar{\kappa}^{\varphi, \zeta}(t_n, x)\Delta t_n}J_{k-1}(x\vert u) + \mathds{1}_{k\geq 1}(k) \frac{(k-1)\Delta t_n}{1-2iu\bar{\kappa}^{\varphi, \zeta}(t_n, x)\Delta t_n}J_{k-2}(x\vert u).
		\end{align}
	\end{lemma}
	\begin{proof}
		Combining \eqref{eq:def:j_k} with \eqref{eq:sde:discretization:decoupled} we get
		\begin{align}
			J_k(x\vert u) = \int_{\mathds{R}} \exp(iu(x+\bar{m}^{\varphi, \zeta}(t_n, x)\Delta t_n + \bar{s}^{\varphi, \zeta}(t_n, x)\xi + \bar{\kappa}^{\varphi, \zeta}(t_n, x)\xi^2))\xi^k \frac{\exp(-\xi^2/(2\Delta t_n))}{\sqrt{2\pi\Delta t_n}}\mathrm{d}\xi,
		\end{align}
        as $\Delta W_n \sim \mathcal{N}(0, \Delta t_n)$.
		It is straightforward to check that
		\begin{align}
			\xi^k\exp(-\xi^2/(2\Delta t_n))=\mathds{1}_{k\geq 1}(k)(k-1)\Delta t_n\xi^{k-2}\exp(-\xi^2/(2\Delta t_n)) - \Delta t_n \frac{\mathrm{d}(\xi^{k-1}\exp(-\xi^2/(2\Delta t_n))}{\mathrm{d}\xi}.
		\end{align}
		Plugging this in above gives
		\begin{align}\label{eq:proof:type2:step1}
			J_k(x\vert u) = \begin{aligned}[t]
				&\mathds{1}_{k\geq 1}(k)(k-1)\Delta t_n J_{k-2}(x\vert u)\\
                &\begin{aligned}[t]
                    -\frac{\Delta t_n}{\sqrt{2\pi\Delta t_n}}\int_{\mathds{R}} &\exp(iu(x+\bar{m}^{\varphi, \zeta}(t_n, x)\Delta t_n + \bar{s}^{\varphi, \zeta}(t_n, x)\xi+ \bar{\kappa}^{\varphi, \zeta}(t_n, x)\xi^2))\\
				&\times\frac{\mathrm{d}(\xi^{k-1}\exp(-\xi^2/(2\Delta t_n))}{\mathrm{d}\xi}\mathrm{d}\xi\eqqcolon\mathds{1}_{k\geq 1}(k) (k-1)\Delta t_n J_{k-2}(x\vert u) + I_k(x\vert u).
                \end{aligned}
			\end{aligned}
		\end{align}
		We apply integration by parts on the second term, which gives
		\begin{align}\label{eq:proof:type2:step2}
			I_k(x\vert u) &= \begin{aligned}[t]
				&-\Delta t_n\left[\exp(iu(x+\bar{m}^{\varphi, \zeta}(t_n, x)\Delta t_n + \bar{s}^{\varphi, \zeta}(t_n, x)\xi+ \bar{\kappa}^{\varphi, \zeta}(t_n, x)\xi^2))\xi^{k-1}\frac{\exp(-\xi^2/(2\Delta t_n))}{\sqrt{2\pi\Delta t_n}}\right]_{-\infty}^{+\infty}\\
				&+\Delta t_n\int_{\mathds{R}} iu\bar{s}^{\varphi, \zeta}(t_n, x)\exp(iuX_{n+1}^{\pi, \varphi, \zeta}(\xi))\xi^{k-1}\frac{\exp(-\xi^2/(2\Delta t_n))}{\sqrt{2\pi\Delta t_n}}\mathrm{d}\xi\\
				&+ \Delta t_n\int_{\mathds{R}}2iu\bar{\kappa}^{\varphi, \zeta}(t_n, x)\exp(iuX_{n+1}^{\pi, \varphi, \zeta}(\xi))\xi^{k}\frac{\exp(-\xi^2/(2\Delta t_n))}{\sqrt{2\pi\Delta t_n}}\mathrm{d}\xi
			\end{aligned}\\
			&\equiv 0 + iu\bar{s}^{\varphi, \zeta}(t_n, x)\Delta t_n J_{k-1}(x\vert u) + 2iu\bar{\kappa}^{\varphi, \zeta}(t_n, x)\Delta t_n J_k(x\vert u),
		\end{align}
        for any $v$ with sufficient radial decay in space.
		Combining \eqref{eq:proof:type2:step1} and \eqref{eq:proof:type2:step2} finally gives
		\begin{align}
			J_k(x\vert u) = \frac{iu\bar{s}^{\varphi, \zeta}(t_n, x)\Delta t_n}{1-2iu\bar{\kappa}^{\varphi, \zeta}(t_n, x)\Delta t_n}J_{k-1}(x\vert u)+\mathds{1}_{k\geq 1(k)}\frac{(k-1)\Delta t_n}{1-2iu\bar{\kappa}^{\varphi, \zeta}(t_n, x)\Delta t_n}J_{k-2}(x\vert u).
		\end{align}
	\end{proof}
	Combining this with the result of lemma \ref{lemma:j_k}, introducing the notation
    \begin{align}\label{def:Phi}
        \Phi^{\pi, \varphi, \zeta}_n(u\vert x)\coloneqq \phi^{\pi, \varphi, \zeta}_n(u\vert x)\exp(-iua),
    \end{align}
    we subsequently gather via recursion
	\begin{align}
		\mathds{E}_n^x[h(t_{n+1}, X_{n+1}^{\pi, \varphi, \zeta})] &= \sideset{}{'}\sum_{l=0}^{\infty} \mathcal{H}_l(t_{n+1}) \Re{\Phi^{\pi, \varphi, \zeta}_n(l\pi/(b-a)\vert x)},\\
		\mathds{E}_n^x[h(t_{n+1}, X_{n+1}^{\pi, \varphi, \zeta})\Delta W_n] &= \sideset{}{'}\sum_{l=0}^{\infty} \mathcal{H}_l(t_{n+1}) \Re{\frac{il\pi\bar{s}^{\varphi, \zeta}(t_n, x)\Delta t_n/(b-a)}{1-2il\pi\bar{\kappa}^{\varphi, \zeta}(t_n, x)\Delta t_n/(b-a)}\Phi^{\pi, \varphi, \zeta}_n(l\pi/(b-a)\vert x)},\\
		\mathds{E}_n^x[h(t_{n+1}, X_{n+1}^{\pi, \varphi, \zeta})(\Delta W_n)^2] &= \begin{aligned}[t]
			\sideset{}{'}\sum_{l=0}^{\infty} \mathcal{H}_l(t_{n+1}) \Bigg[&\Re{\left(\frac{il\pi\bar{s}^{\varphi, \zeta}(t_n, x)\Delta t_n/(b-a)}{1-2il\pi\bar{\kappa}^{\varphi, \zeta}(t_n, x)\Delta t_n/(b-a)}\right)^2\Phi^{\pi, \varphi, \zeta}_n(l\pi/(b-a)\vert x)}\\
			&\Re{\frac{(k-1)\Delta t_n}{1-2il\pi\bar{\kappa}^{\varphi, \zeta}(t_n, x)\Delta t_n/(b-a)}\Phi^{\pi, \varphi, \zeta}_n(l\pi/(b-a)\vert x)}\Bigg].
		\end{aligned}
	\end{align}
	Given the analytical expression for $\Phi_n^{\pi, \varphi, \zeta}(u\vert x)$ established by lemma \ref{lemma:chf}, a sufficiently truncated finite cosine expansion gives the necessary COS estimates. Consequently, using lemma \ref{lemma:j_k} and the expressions above, one can compute all conditional expectations arising in the scheme \eqref{eq:zhao_scheme}, provided that the Fourier cosine expansion coefficients of the deterministic functions $z(t_{n+1}, x), y(t_{n+1}, x)$ and $\bar{f}^{z(t_{n+1}, \cdot), y(t_{n+1}, \cdot)}(t_{n+1}, x)$ are available, or at least can be approximated.

    Nonetheless, in order to make the scheme \eqref{eq:crisan_scheme} fully-implementable in the BCOS framework, one needs to establish an additional integration by parts formula that allows the computations of conditional expectations of the form
    \begin{align}\label{eq:crisan:conditional_expectation}
        \mathbb{E}_n^x\Big[\big(\int_{t_{n}}^{t_{n+1}} (s-t_n)\mathrm{d}W_s\big)v(t_{n+1}, X_{n+1}^{\pi, \varphi, \zeta})\Big],
    \end{align}
    appearing in \eqref{eq:crisan_scheme}. In the following proposition, we show that for given choices of $\theta_1, \theta_2, \theta_3, \theta_4$ the second-order scheme of \cite{crisan_second_2014} in \eqref{eq:crisan_scheme} is included in the generalized theta-scheme of \cite{zhao_generalized_2012} in \eqref{eq:zhao_scheme}.
    \begin{proposition}\label{prop:crisan_in_zhao}
        For conditional expectations of the form \eqref{eq:crisan:conditional_expectation}, the following identity holds
        \begin{align}\label{eq:crisan:integration-by-parts:result}
        \mathbb{E}_n^x\Big[\big(\int_{t_{n}}^{t_{n+1}} s\mathrm{d}W_s\big)v(t_{n+1}, X_{n+1}^{\pi, \varphi, \zeta})\Big]=\frac{\Delta t_n}{2}\mathbb{E}_n^x\big[\Delta W_n v(t_{n+1}, X_{n+1}^{\pi, \varphi, \zeta})\big],
    \end{align}
    with $\Delta B_n\coloneqq \int_{t_{n}}^{t_{n+1}} (s-t_n)\mathrm{d}W_s$. In particular, \eqref{eq:crisan_scheme} is included in \eqref{eq:zhao_scheme} for $\theta_1=1/2, \theta_3=1-\theta_2, \theta_4=0$.
    \end{proposition}
    \begin{proof}
        Let us consider the conditional expectation given by \eqref{eq:crisan:conditional_expectation}. Similarly to lemma \ref{lemma:j_k}, we use an integration by parts argument to show that
    \begin{align}\label{eq:crisan:integration-by-parts}
        \mathbb{E}_n^x[\Delta B_n v(t_{n+1}, X_{n+1}^{\pi, \varphi, \zeta})]=\frac{\mathbb{C}\textnormal{ov}[\Delta B_n, \Delta W_n]}{\mathbb{V}\textnormal{ar}[\Delta W_n]}\mathbb{E}_n^x[\Delta W_n v(t_{n+1}, X_{n+1}^{\pi, \varphi, \zeta})].
    \end{align} 
    As the integrand in $\Delta B_n\coloneqq \int_{t_{n}}^{t_{n+1}} (s-t_n)\mathrm{d}W_s$ is deterministic, we have that $\Delta B_n\sim\mathcal{N}(0, (\Delta t_n)^3/3=\int_{t_{n}}^{t_{n+1}} (s-t_n)^2\mathrm{d}s)$, and due to It\^{o}'s isometry
    \begin{align}\label{eq:crisan:moments}
        \mathbb{E}[\Delta B_n]=0,\quad \mathbb{E}[(\Delta B_n)^2]=(\Delta t_n)^3/3,\quad \mathbb{C}\textnormal{ov}[\Delta B_n, \Delta W_n]=(\Delta t_n)^2/2.
    \end{align}
    Then the joint distribution of $(\Delta B_n, \Delta W_n)$ is a multivariate normal distribution with covariance matrix $\Sigma=\begin{pmatrix}
            (\Delta t_n)^3/3 & (\Delta t_n)^2/2\\
            (\Delta t_n)^2/2 & \Delta t_n
        \end{pmatrix}$ and $\det(\Sigma)=(\Delta t_n)^4/12$.
        Consequently, the conditional expectation in \eqref{eq:crisan:conditional_expectation} takes the following form
        \begin{align}\label{eq:appendix:integration-by-parts:step1}
            \int_{-\infty}^\infty\int_{-\infty}^\infty v(t_{n+1}, X_{n+1}^{\pi, \varphi, \zeta}(\xi))\eta\frac{\exp(-\frac{\Delta t_n\eta^2 -(\Delta t_n)^2 \eta\xi + (\Delta t_n)^3\xi^2/3}{2})}{2\pi (\Delta t_n)^4/12}\mathrm{d}\eta\mathrm{d}\xi.
        \end{align}
        By formal differentiation, it is straightforward to show that
        \begin{align}
            \frac{\partial}{\partial\eta}\Bigg(\frac{-\exp(-\frac{\Delta t_n\xi^2 -(\Delta t_n)^2 \eta\xi + (\Delta t_n)^3\xi^2/3}{2})}{\Delta t_n}\Bigg)=\begin{aligned}[t]
                &\eta\exp(-\frac{\Delta t_n\xi^2 -(\Delta t_n)^2 \eta\xi + (\Delta t_n)^3\xi^2/3}{2})\\
                &-\frac{\Delta t_n}{2}\xi\exp(-\frac{\Delta t_n\xi^2 -(\Delta t_n)^2 \eta\xi + (\Delta t_n)^3\xi^2/3}{2}),
            \end{aligned}
        \end{align}
        and thus we can write \eqref{eq:appendix:integration-by-parts:step1} in the following way, by virtue of integration by parts and the Fubini theorem
        \begin{align}
            \begin{aligned}[t]
                \int_{-\infty}^\infty &\frac{v(t_{n+1}, X_{n+1}^{\pi, \varphi, \zeta}(\xi))}{2\pi(\Delta t_n)^4/12}\Bigg[\frac{-\exp(-\frac{\Delta t_n\xi^2 -(\Delta t_n)^2 \eta\xi + (\Delta t_n)^3\xi^2/3}{2})}{\Delta t_n}\Bigg]_{-\infty}^{\infty}\mathrm{d}\xi\\
                &+\frac{\Delta t_n}{2}\int_{-\infty}^\infty v(t_{n+1}, X_{n+1}^{\pi, \varphi, \zeta}(\xi))\xi \frac{\exp(-\frac{\Delta t_n\eta^2 -(\Delta t_n)^2 \eta\xi + (\Delta t_n)^3\xi^2/3}{2})}{2\pi (\Delta t_n)^4/12}\mathrm{d}\eta\mathrm{d}\xi.
            \end{aligned}
        \end{align}
        The integrand of the first term vanishes for any $v$ with sufficient spatial radial decay, whereas the second term is identically equal to the right-hand side of \eqref{eq:crisan:integration-by-parts} combined with \eqref{eq:crisan:moments}.
        This completes the proof of identity \eqref{eq:crisan:integration-by-parts:result}.
        
        Combining this with the second line of \eqref{eq:crisan_scheme}, we finally find that 
    \eqref{eq:crisan_scheme} is also included in the generalized $\theta$-scheme of \cite{zhao_generalized_2012} given by \eqref{eq:zhao_scheme}, with the particular choice of $\theta_1=1/2, \theta_3=1-\theta_2, \theta_4=0$.
    \end{proof}

     In light of proposition \ref{prop:crisan_in_zhao}, for the upcoming numerical experiments, we implement the theta-scheme \eqref{eq:zhao_scheme} for general $\theta$s, which includes the second-order scheme in \eqref{eq:crisan_scheme} for the choices above.

    At this point the COS approximations of the conditional expectations in both \eqref{eq:zhao_scheme} and \eqref{eq:crisan_scheme} can be obtained as follows
    \begin{align}
        z(t_n, x)&= \begin{aligned}[t]
            \frac{1}{\theta_3 \Delta t_n}\Bigg(&\theta_4\Delta t_n \sideset{}{'}\sum_{k=0}^{K-1} \mathcal{Z}_k(t_{n+1})\Re{\Phi_n^{\pi, \varphi, \zeta}\big(k/(b-a)|x\big)}\\
            &+\sideset{}{'}\sum_{k=0}^{K-1}\begin{aligned}[t]
                \big((\theta_3-\theta_4)&\mathcal{Y}_k(t_{n+1})+(1-\theta_2)\Delta t_n \bar{\mathcal{F}}_k(t_{n+1})\big)\\
                \times&\Re{\frac{ik\pi\bar{s}^{\varphi, \zeta}(t_n, x)\Delta t_n/(b-a)}{1-2ik\pi\bar{\kappa}^{\varphi, \zeta}(t_n, x)\Delta t_n/(b-a)}\Phi^{\pi, \varphi, \zeta}_n(k\pi/(b-a)|x)}\Bigg),
            \end{aligned}
        \end{aligned}\label{eq:abstract_bcos:z}\\
        y(t_n, x) &= \begin{aligned}[t]
            &\theta_1\Delta t_n f(t_n, x, y(t_n, x), z(t_n, x)) \\&+ \sideset{}{'}\sum_{k=0}^{K-1}(\mathcal{Y}_k(t_{n+1}) + (1-\theta_1)\Delta t_n \bar{\mathcal{F}}_k(t_{n+1}))\Re{\Phi_n^{\pi, \varphi, \zeta}\big(k/(b-a)|x\big)}.\label{eq:abstract_bcos:y}
        \end{aligned}
    \end{align}
    However, these expressions are only fully implementable in the decoupled framework, i.e. when the drift and diffusion coefficients in \eqref{eq:fbsde:sde} do not depend on the solution pair of the backward equation. Indeed, in that case the characteristic function does not depend on the decoupling pair $\varphi, \zeta$ either -- see \eqref{eq:sde:characteristic_function}. Nonetheless, in the coupled framework, one needs to find reasonable approximations of the true decoupling relations given in \eqref{eq:sde:decoupled}, and use those for $(\varphi, \zeta)$ in between two time steps $t_n$ and $t_{n+1}$. In the paper of Huijskens et al. \cite{huijskens_efficient_2016}, using an Euler scheme for the discretization of the forward component -- see \eqref{eq:coefficients:euler}, in particular $\bar{\kappa}^{\varphi, \zeta}(t_n, x)\equiv 0$ -- there are three choices made in this regard. In the so called explicit method, argued by sufficient continuity in time of the solution pair of \eqref{eq:fbsde:bsde}, the decoupling relations are chosen to be the discrete time approximations at the next point in time $\varphi=y(t_{n+1}, \cdot), \zeta=z(t_{n+1}, \cdot)$. Note that \eqref{eq:zhao_scheme} is a backward recursion, therefore these approximations are indeed available at the processing of time step $t_n$. Alternatively, \cite{huijskens_efficient_2016} proposed a so-called local method, where -- starting off from some initial conditions -- the decoupling relations $(\varphi, \zeta)$ are gathered through Picard iterations at each time step $t_n$, assuming that the mapping $(\varphi, \zeta)\mapsto (y(t_{n+1}, \cdot), z(t_{n+1}, \cdot))$ defined by \eqref{eq:abstract_bcos:z} is a contraction for small enough time steps. Finally, \cite{huijskens_efficient_2016} also proposed a so called global approach, in which the solution pair is gathered through taking Picard iterations over the whole backward recursion in \eqref{eq:zhao_scheme} -- similar to \cite{bender_time_2008}. As found in \cite{huijskens_efficient_2016}, the most efficient of these three options is the explicit method, as it only requires computing $\Phi_n^{\pi, \varphi, \zeta}(k\pi/(b-a)| x)$ once at every time step, and it does not require the mapping in \eqref{eq:abstract_bcos:z} to be a contraction either. In our implementation, we confirmed the findings of \cite{huijskens_efficient_2016} even in case of Milstein- and 2.0 weak Taylor discretizations, and found that a local method is both less stable and less accurate while creating an unnecessary computational overhead compared to the explicit decoupling.
    Therefore, in what follows we consider the explicit method only.

    \subsection{Derivative approximations and the BCOS method with higher-order Taylor schemes}

    Nonetheless, in contrast to \cite{huijskens_efficient_2016}, the BCOS approximations \eqref{eq:abstract_bcos:y}-\eqref{eq:abstract_bcos:z} are not implementable, in case the forward diffusion is discretized with a second-order Taylor approximation, including the Milstein- and 2.0 weak Taylor approximations. Indeed, then the corresponding second-order term in \eqref{eq:sde:discretization:decoupled} has $\bar{\kappa}^{\varphi, \zeta}\neq 0$. In particular, in case of the Milstein scheme, through the total derivative of the diffusion coefficient, the derivatives of the decoupling relations also appear in the computation of $\bar{\kappa}^{\varphi, \zeta}$ which implies that one needs to choose decoupling relations carefully such that their derivatives are either available in closed form or can be accurately approximated. Our main result is built on the insight that whenever the decoupling relations in the BCOS method are chosen according to the explicit scheme -- that is $\varphi=y(t_{n+1}, \cdot), \zeta=z(t_{n+1}, \cdot)$ -- the derivatives $\partial_x \varphi$ and $\partial_x \zeta$ are analytically available provided by the smoothness of the Fourier cosine expansion \eqref{eq:def:cosine_expansion_coefficients}.
    In fact, when the Fourier cosine expansion coefficients $\mathcal{Y}_k(t_{n+1}), \mathcal{Z}_k(t_{n+1})$ of $y(t_{n+1}, \cdot), z(t_{n+1}, \cdot)$ in \eqref{eq:abstract_bcos:y} and \eqref{eq:abstract_bcos:z} are known, then
    \begin{align}\label{eq:derivative_approximations:0}
    \begin{split}
        y(t_{n+1}, x)=\sideset{}{'}\sum_{k=0}^{K-1}\mathcal{Y}_k(t_{n+1})\cos(\frac{k\pi}{b-a}(x-a)),\quad z(t_{n+1}, x) = \sideset{}{'}\sum_{k=0}^{K-1} \mathcal{Z}_k(t_{n+1})\cos(\frac{k\pi}{b-a}(x-a)),
    \end{split}
    \end{align}
    and their corresponding derivatives can analytically be computed as follows
    \begin{align}\label{eq:derivative_approximations:1}
    \begin{split}
        \partial_x y(t_{n+1}, x)&=\sideset{}{'}\sum_{k=0}^{K-1} -\frac{k\pi}{b-a}\mathcal{Y}_k(t_{n+1})\sin(\frac{k\pi}{b-a}(x-a)),\\
        \partial_x z(t_{n+1}, x)&=\sideset{}{'}\sum_{k=0}^{K-1} -\frac{k\pi}{b-a}\mathcal{Z}_k(t_{n+1})\sin(\frac{k\pi}{b-a}(x-a)).
    \end{split}
    \end{align}
    Similarly, the second derivatives appearing in $\bar{m}^{\varphi, \zeta}$ and $\bar{s}^{\varphi, \zeta}$ in \eqref{eq:coefficients:2.0-weak-taylor} are given by 
    \begin{align}\label{eq:derivative_approximations:2}
    \begin{split}
        \partial_{xx}^2 y(t_{n+1}, x)&=\sideset{}{'}\sum_{k=0}^{K-1} -\left(\frac{k\pi}{b-a}\right)^2\mathcal{Y}_k(t_{n+1})\cos(\frac{k\pi}{b-a}(x-a)),\\
        \partial_{xx}^2 z(t_{n+1}, x)&=\sideset{}{'}\sum_{k=0}^{K-1} -\left(\frac{k\pi}{b-a}\right)^2\mathcal{Z}_k(t_{n+1})\cos(\frac{k\pi}{b-a}(x-a)).
    \end{split}
    \end{align}
    For the BCOS approximation at time step $t_n$, we choose to decouple the Markov transition \eqref{eq:sde:discretization:decoupled} of the forward diffusion by setting $\varphi=y(t_{n+1}, \cdot)$ and $\zeta=z(t_{n+1}, \cdot)$, in a similar fashion to the explicit method in \cite{huijskens_efficient_2016}. In case the Markov transition is modeled by the Euler scheme \eqref{eq:coefficients:euler}, the Fourier cosine expansions in \eqref{eq:derivative_approximations:0} are sufficient to compute \eqref{def:Phi} and thus make the abstract scheme in \eqref{eq:abstract_bcos:y}-\eqref{eq:abstract_bcos:z} fully implementable. However, when the transition is approximated by a Milstein scheme then the coefficients $\bar{m}^{\varphi, \zeta}$ and $\bar{s}^{\varphi, \zeta}$ depend on the derivatives $\partial_x y(t_{n+1}, \cdot)$ and $\partial_x z(t_{n+1}, \cdot)$ via the total derivative of $\partial_x \bar{\sigma}^{y(t_{n+1}, \cdot), z(t_{n+1}, \cdot)}(t_n, x)$ according to \eqref{eq:total_derivative_in_x}. Nonetheless, due to \eqref{eq:derivative_approximations:1}, these expressions can be analytically computed for the choice $\varphi=y(t_{n+1}, \cdot)$ and $\zeta=z(t_{n+1}, \cdot)$, which makes the computations of the coefficients in \eqref{eq:coefficients:milstein} possible, and subsequently enables to compute $\Phi_n^{\pi, \varphi, \zeta}(u|x)$ even in the case of a Milstein transition.

    Moreover, when the transition is given by a simplified order 2.0 weak Taylor approximation \eqref{eq:coefficients:2.0-weak-taylor}, then the coefficients $\bar{m}^{\varphi, \zeta}$ and $\bar{s}^{\varphi, \zeta}$ also depend on the second order derivatives of the decoupling fields $\varphi$ and $\zeta$, through $\partial_{xx}^2 \bar{\sigma}^{\varphi, \zeta}(t_n, x)$ according to \eqref{eq:total_derivative_in_x}.  Similarly to the previous reasoning, when the decoupling is chosen by $\varphi=y(t_{n+1}, \cdot)$ and $\zeta=z(t_{n+1}, \cdot)$, the formulas for the second-order derivatives in \eqref{eq:derivative_approximations:2} enable the computation of the coefficients in \eqref{eq:coefficients:2.0-weak-taylor}. Thereafter, \eqref{eq:derivative_approximations:1} and \eqref{eq:derivative_approximations:2} together enable us to compute the characteristic function in \eqref{def:Phi} even in case the Markov transition is modeled by a 2.0 weak Taylor approximation.

    \paragraph{Coefficient recovery.} The discussion above implicitly relied on the availability of the Fourier cosine expansion coefficients $\mathcal{Y}_k(t_{n+1})$ and $\mathcal{Z}_k(t_{n+1})$, defined by \eqref{eq:def:cosine_expansion_coefficients}. However, during the backward recursion of the discrete approximation of the BSDE as in \eqref{eq:abstract_bcos:y}-\eqref{eq:abstract_bcos:z}, these coefficients cannot be analytically computed. In order to make the transitions implementable, one needs to \emph{numerically approximate} the continuous integral in \eqref{eq:def:cosine_expansion_coefficients}, which can be achieved by a discrete Fourier cosine transform. In order to do this, we construct a discrete spatial grid partitioning the truncated integration range $[a, b]$ defined by the following points
    \begin{align}\label{eq:dct:discretization}
        x_l\coloneqq a+(l+1/2)\frac{b-a}{K},\quad l=0, \dots, K-1,
    \end{align}
    given that the Fourier cosine expansion is truncated to $K$-many terms. We put $\Pi\coloneqq \{x_l: l=0, \dots, K-1\}$. Subsequently, the continuous integral in \eqref{eq:def:cosine_expansion_coefficients} can be approximated by
    \begin{align}\label{eq:dct_approximation}
        \mathcal{V}_k(t_{n+1})\approx \frac{2}{K}\sum_{l=0}^{K-1} v(t_{n+1}, x_l)\cos(k\pi\frac{2l+1}{2K}),
    \end{align}
    which is a Discrete Cosine Transform (DCT) approximation of type 2.
    Given the terminal condition of the BSDE in \eqref{eq:fbsde:bsde}, the coefficient recovery defined by \eqref{eq:dct_approximation} makes the backward recursion in \eqref{eq:abstract_bcos:y}-\eqref{eq:abstract_bcos:z} implementable for all forward transition schemes \eqref{eq:coefficients:euler}-\eqref{eq:coefficients:milstein}-\eqref{eq:coefficients:2.0-weak-taylor}, together with \eqref{eq:derivative_approximations:1}-\eqref{eq:derivative_approximations:2} and the discussion above.

    \paragraph{Picard iterations.} Finally, we address the implicitness of the discrete time approximations in \eqref{eq:abstract_bcos:y} in case $\theta_1>0$. When the time step $\Delta t_n$ is sufficiently small, the implicit equation is uniquely defined for each $x$. In order to find the unique fixed point, we carry out Picard iterations indexed by $p$. We initialize the Picard iterations by the explicit approximation of \eqref{eq:abstract_bcos:y} with $\theta_1=0$
    \begin{align}\label{eq:abstract_bcos:explicit}
        h(t_{n}, x)= \sideset{}{'}\sum_{k=0}^{K-1}(\mathcal{Y}_k(t_{n+1}) + \Delta t_n \bar{\mathcal{F}}_k(t_{n+1}))\Re{\Phi_n^{\pi, \varphi, \zeta}\big(k/(b-a)|x\big)},
    \end{align}
    and set $y^{(0)}(t_{n+1}, x)=h(t_{n}, x)$. From thereon, the following Picard update is carried out
    \begin{align}\label{eq:abstract_bcos:picard_update}
        y^{(p+1)}(t_{n}, x) = y^{(p)}(t_{n}, x) + h(t_{n}, x)
    \end{align}
    until a required error tolerance level $\varepsilon$ is reached, measured by
    \begin{align}\label{eq:abstract_bcos:tolerance_level}
        \max_{x\in \Pi} \abs{y^{(p+1)}(t_n, x) - y^{(p)}(t_n, x)}&\leq \varepsilon.
    \end{align}

    In order to simplify the notation, we define $y_n^\pi(x) \coloneqq y(t_{n+1}, x)$ and $z_n^\pi(x)\coloneqq z(t_{n+1}, x)$ for the series expansions \eqref{eq:derivative_approximations:0}, with expansion coefficients recovered by DCT.
    With the above steps, the coupled BCOS method using higher-order Taylor discretizations as in \eqref{eq:sde:general_second_order_scheme} is now fully-implementable. The complete algorithm is given in algorithm \ref{algorithm}.
    
    \begin{algorithm}[t]
\caption{Coupled BCOS algorithm with higher-order Taylor schemes}\label{algorithm}
\begin{algorithmic}[1]
\Require $[a, b]$: integration range; $K$: number of Fourier terms; $\pi$: discrete time partition of $[0, T]$ 
\Require forward discretization scheme (\eqref{eq:coefficients:euler} or \eqref{eq:coefficients:milstein} or \eqref{eq:coefficients:2.0-weak-taylor}); $\theta_1, \theta_2, \theta_3, \theta_4$: generalized $\theta$-scheme in \eqref{eq:zhao_scheme}
\State $x_i=a+(i+1/2)(b-a)/K, i=0, \dots, K-1$: spatial grid to compute DCT
\If{$\sigma$ depends on $Z$}
    \State Solve $z_N^\pi(x_i)=\sigma(T, x_i, g(x_i), z_N^\pi(x_i))$ for all $i=0, \dots, N-1$
\EndIf
\State $\mathcal{Y}_k(T), \mathcal{Z}_k(T), \bar{\mathcal{F}}_k(T)$\Comment{Fourier coefficients at $T$ -- analytically or via DCT}
\For{$n=N-1, \dots, 0$}
    \State $\varphi \gets y_{n+1}^\pi, \zeta\gets z_{n+1}^\pi$ \Comment{decouple by approximations at next time step}
    \State $\Phi_n^{\pi, y_{n+1}^\pi, z_{n+1}^\pi}$ \Comment{compute characteristic function by \eqref{def:Phi} and \eqref{eq:sde:characteristic_function} with \eqref{eq:coefficients:euler} or \eqref{eq:coefficients:milstein} or \eqref{eq:coefficients:2.0-weak-taylor}}
    \State $z_n^\pi(x_i)$ \Comment{compute $Z$ approximations over spatial grid by \eqref{eq:abstract_bcos:z}}

    \State $y_n^{(0)}\gets h(t_n, x)$ \Comment{initialize Picard iterations by explicit approximation in \eqref{eq:abstract_bcos:explicit}}
    
    \While{($p\leq\text{max. Picard iter.}$) \textbf{and} (tolerance in \eqref{eq:abstract_bcos:tolerance_level} is not reached)}
        \State $y_n^{(p+1)}(x_i)$ \Comment{Picard update according to \eqref{eq:abstract_bcos:picard_update}}
    \EndWhile
    \State $y_n^\pi(x_i)\gets y_n^{p+1}(x_i)$

    \State $\mathcal{Y}_k(t_n), \mathcal{Z}_k(t_n), \bar{\mathcal{F}}_k(t_n)$ \Comment{coefficient recovery by DCT \eqref{eq:dct_approximation} on $z_n(x_i), y_n^\pi(x_i), \bar{f}^{y_n^\pi, z_n^\pi}(t_n, x_i)$}
\EndFor
\end{algorithmic}
\end{algorithm}

    \subsection{Errors and computational complexity}\label{sec:computational_complexity}
    In what follows, we discuss the errors and computational complexity induced by the coupled BCOS method in algorithm \ref{algorithm}.
    \paragraph{Error analysis.} The main sources of numerical errors and their contributions to the final approximation accuracy can be summarized as follows
    \begin{itemize}
        \item $K$ -- truncation of the Fourier cosine series: for smooth densities, the Fourier cosine expansion terms converge exponentially -- see e.g. \cite{fang_novel_2009}. However, as algorithm \ref{algorithm} relies on Discrete Cosine Transforms to recover the coefficients $\mathcal{Y}_k(t_n), \mathcal{Z}_k(t_n), \mathcal{F}(t_n)$, the total error term only converges quadratically in the number of Fourier terms $\mathcal{O}(K^{-2})$, which is the accuracy of the numerical integration in \eqref{eq:dct_approximation};
        \item $N$ -- time discretization:
        \begin{itemize}
            \item forward SDE \eqref{eq:fbsde:sde}: depending on whether the forward transition is approximated by an Euler \eqref{eq:coefficients:euler}, Milstein \eqref{eq:coefficients:milstein} or 2.0 weak Taylor scheme \eqref{eq:coefficients:2.0-weak-taylor}, weak and strong errors converge according to the rates collected in table \ref{tab:convergence-rates};
            \item backward SDE \eqref{eq:fbsde:bsde}: considering the generalized $\theta$-scheme of \cite{zhao_generalized_2012}, the corresponding discrete time approximations in \eqref{eq:zhao_scheme} converge with $\mathcal{O}(h^2)$ in the strong sense when $\theta_1=\theta_2=\theta_3=1/2$, $\theta_4\leq \abs{\theta_3}$. This in particular includes the scheme of \cite{crisan_second_2014}, for $\theta_4=0$ as established by \eqref{eq:crisan:integration-by-parts:result};
        \end{itemize}
        \item $P$ -- Picard iterations to approximate the implicit part of the conditional expectation in \eqref{eq:abstract_bcos:y} when $\theta_1>0$: when the driver $f$ is Lipschitz in its spatial arguments, the implicit mapping is contractive for small enough time steps $\Delta t_n$; then the Picard iterations converge exponentially, with the constant depending on the Lipschitz constants of the driver $f$ in \eqref{eq:fbsde:bsde};  
        \item $a, b$ -- truncated integration range: for each problem the range should be chosen carefully wide enough such that its overall contribution is negligible. The precise impact of this error is difficult to quantify; a recent result on the optimal choice of $a, b$ is given in \cite{junike_precise_2022}, however, it is not straightforward to extend this proof to coupled FBSDEs as one can not have a-priori guarantees about the distribution of $X$ at a given time due to the coupling. In the numerical experiments below, we choose a wide enough integration range such that the corresponding error term is negligible.
    \end{itemize}
    We emphasize that the derivative approximations \eqref{eq:derivative_approximations:1} and \eqref{eq:derivative_approximations:2}, which enable second-order Taylor discretizations \eqref{eq:sde:general_second_order_scheme} of the forward SDE do not add numerical errors, as they can be computed analytically given the cosine expansion coefficients $\mathcal{Y}_k(t_n), \mathcal{Z}_k(t_n)$ at each time step. This is the key observation that extends the coupled BCOS method to Milstein and 2.0 weak Taylor approximations for the Markov transition, together with the closed-form characteristic function provided by lemma \ref{lemma:chf}.
    
    \paragraph{Computational complexity.} The overall computational complexity of the second-order scheme proposed in algorithm \ref{algorithm} consists of the following components while processing time step $t_n$:
    \begin{itemize}
        \item $\mathcal{Y}_k(T), \mathcal{Z}_k(T), \bar{\mathcal{F}}_k(t_n)$ -- terminal expansion coefficients: analytically $\mathcal{O}(K)$ or by DCT $\mathcal{O}(K\log(K))$;
        \item $\varphi\gets y_{n+1}^\pi, \zeta\gets z_{n+1}^\pi$ -- computing decoupling fields: one needs to compute \eqref{eq:derivative_approximations:0}, \eqref{eq:derivative_approximations:1} and \eqref{eq:derivative_approximations:2} for each pair $(x_i, k)$ leading to\\
        \begin{tabular}{ccc}
            Euler & Milstein & 2.0-weak-Taylor \\
            $\mathcal{O}(K)$ & $\mathcal{O}(K^2)$ & $\mathcal{O}(K^2)$
        \end{tabular}
        \item $\Phi_n^{\pi, \varphi, \zeta}(k\pi/(b-a)\vert x_i)$ -- computation of the characteristic function over the spatial grid $x_i$ for each $k=0, \dots, K-1$: $\mathcal{O}(K^2)$;
        \item $y_n^{(p+1)}$ -- Picard iterations for the implicit part in \eqref{eq:abstract_bcos:y}: $\mathcal{O}(PK)$
        \item $y_n^\pi(x_i), z_n^\pi(x_i)$ -- BCOS approximations \eqref{eq:abstract_bcos:y} and \eqref{eq:abstract_bcos:z}: $\mathcal{O}(K\log(K))$;
        \item $\mathcal{Y}_k(t_n), \mathcal{Z}_k(t_n), \bar{\mathcal{F}}_k(t_n)$ -- coefficient recovery by DCT: $\mathcal{O}(K\log(K))$.
    \end{itemize}
    Apart from the first item, every point is repeated for all time steps $n=N-1, \dots, 0$, and the coupled BCOS method scales linearly with respect to the number of discretization points in time. The overall computational complexity is thus given by $\mathcal{O}(N(N+N^2+PN+N\log(N)))$. Comparing this to the computational complexity of the \textit{explicit method} in \cite{huijskens_efficient_2016}, which uses the Euler discretization \eqref{eq:coefficients:euler} in the approximation of the Markov transition \eqref{eq:sde:general_second_order_scheme}, we find that our generalized higher-order BCOS method admits to the same computational complexity with a higher constant. In fact, our extensions to the method using higher-order Taylor schemes such as the Milstein \eqref{eq:coefficients:milstein} or 2.0 weak Taylor \eqref{eq:coefficients:2.0-weak-taylor} schemes only create a computational overhead in the computations of the derivatives in \eqref{eq:derivative_approximations:1} and \eqref{eq:derivative_approximations:2}, which  scale as $\mathcal{O}(K^2)$ in both cases. As the most expensive part of the BCOS algorithm is to compute the characteristic function \eqref{eq:sde:characteristic_function} at each time step, which requires $\mathcal{O}(K^2)$ operations, the additional steps for the Milstein and 2.0 weak Taylor discretizations only double and triple this, respectively. As we shall see in the numerical experiments presented below, this marginal computational overhead is justified by the significantly improved convergence rates -- see table \ref{tab:convergence-rates} -- in the number of discretization points in time. In other words, even though the Milstein and 2.0 weak Taylor approximations require the extra computation of the derivatives of the decoupling fields, they reach a desired error tolerance level faster in $N$ due to their higher strong and weak convergence rates, respectively.

    \begin{remark}
    This paper is concerned with scalar valued FBSDE systems, i.e. the solution to \eqref{eq:fbsde} is triple of scalar valued processes. Let us briefly highlight the main challenges one faces when generalizing algorithm \ref{algorithm} to higher dimensional equations. First, in case \eqref{eq:fbsde:sde} admits a vector valued solution, one needs additional assumptions to preserve the convergence rates in table \ref{tab:convergence-rates} for higher-order Taylor schemes. For instance, the Milstein scheme in \eqref{eq:coefficients:milstein} only has a strong convergence rate of $\mathcal{O}(h)$ for commutative noise -- see \cite{kloeden_numerical_1992}. Moreover, for a vector valued $X$ the corresponding decoupling fields in \eqref{eq:feynman-kac} are multivariate functions, admitting multi-dimensional Fourier cosine expansions in place of \eqref{eq:def:cosine_expansion_coefficients}. In particular, this implies a Fourier expansion along \emph{each dimension}, on top of the discretization \eqref{eq:dct:discretization} and numerical integration \eqref{eq:dct_approximation} of the multi-dimensional domain. Subsequently $\Phi_n^{\pi, \varphi, \zeta}$ in \eqref{def:Phi} becomes a $K^d\times K^d$ matrix. As discussed above, the computation of these transitional weights determined by the characteristic function is the most expensive part of algorithm \ref{algorithm}, which would scale exponentially in the number of spatial dimensions. Regardless of this curse of dimensionality, the method could be extended up to dimension to $3$ with increasing memory constraints on $K$. We refer to \cite{ruijter_two-dimensional_2012} where a two-dimensional COS method is presented, outside of the FBSDE framework.
    \end{remark}

    \section{Numerical experiments}\label{sec:numerical}
    The BCOS method has been implemented in a Python library, which is openly accessible through the following \href{https://github.com/balintnegyesi/coupled-BCOS}{github repository}\footnote{https://github.com/balintnegyesi/coupled-BCOS}. All computations were carried out on a Dell Alienware Aurora R10 machine equipped with an AMD Ryzen 9 3950X CPU (16 cores, 64Mb cache, 4.7 Ghz) using double precision.
    The solution triple of each equation is computed by plugging the analytical solution of the corresponding quasi-linear PDE in \eqref{eq:feynman-kac} into a simplified order 2.0 weak Taylor discretization of the forward diffusion in \eqref{eq:fbsde:sde} over a fine, equidistant time grid, consisting of $N'=10^6$ equally sized intervals. This way the time discretization error of the reference solution is negligible.

    Each FBSDE below is discretized by an equidistant time partition, consisting of $N+1$ points, leading to a uniform time step size $h=T/N$. The forward SDEs are discretized by \eqref{eq:sde:general_second_order_scheme}, including Euler, Milstein and 2.0 weak Taylor approximations. Each BSDE below is discretized according to the generalized $\theta$-scheme in \eqref{eq:zhao_scheme} with $\theta_1=\theta_2=\theta_3=1/2$. In order to distinguish between the two second-order schemes of the backward equation given by \eqref{eq:zhao_scheme} and \eqref{eq:crisan_scheme}, we consider the values $\theta_4=-1/2$ and $\theta_4=0$ corresponding to \eqref{eq:crisan_scheme} as shown by \eqref{eq:crisan:integration-by-parts:result}. For the COS method we specify an integration range $[a, b]$ wide enough for each problem specifically, such that the truncation error is negligible. For the implicit part of \eqref{eq:abstract_bcos:y}, we set the maximum number of Picard iterations in \eqref{eq:abstract_bcos:picard_update} to $100$, and the tolerance level in \eqref{eq:abstract_bcos:tolerance_level} to $\varepsilon=10^{-15}$.\footnote{thanks to the exponential convergence of the Picard iterations, tolerance level is usually reached in $5$ Picard iterations}

    Strong $L^2$ errors are computed over an independently simulated Monte Carlo sample consisting of $M=2^{10}$ paths of the Brownian motion, using the following discrete approximations
    \begin{align}\label{eq:strong-error-computation}
        \begin{split}
            \text{strong error } X &=\max_{n=0, \dots, N} \left(\frac{1}{M}\sum_{m=1}^{M}|X_{n}^{\pi}(m) -X_{t_{n}}(m)|^2\right)^{1/2},\\
        \text{strong error } Y&=\max_{n=0, \dots, N} \left(\frac{1}{M}\sum_{m=1}^{M}|Y_{n}^{\pi}(m) - Y_{t_{n}}(m)|^2\right)^{1/2},\\
        \text{strong error } Z&=\left(\frac{T}{NM}\sum_{n=0}^{N-1}\sum_{m=1}^{M}|Z_{n}^{\pi}(m) - Z_{t_{n}}(m)|^2\right)^{1/2}.
        \end{split}
    \end{align}
    The total strong error is given by the sum of the above three terms.
    Given the deterministic, fixed initial condition in \eqref{eq:fbsde:sde} for each of the equations below, the approximation errors at $t=0$ coincide with the weak errors in mean
    \begin{align}
        \text{error } Y_0 = |y_0^\pi(x_0) - u(0, x_0)|,\quad \text{error } Z_0 = |z_0^\pi(x_0)-v(0, x_0)|.
    \end{align}
    The total weak approximation error at $t_0$ is given by the sum of the above two terms.

    \subsection{Example 1: decoupled FBSDE}
    \begin{figure}[t]
        \centering
        \begin{subfigure}[t]{\sizethreebyone\textwidth}
        \includegraphics[width=\linewidth]{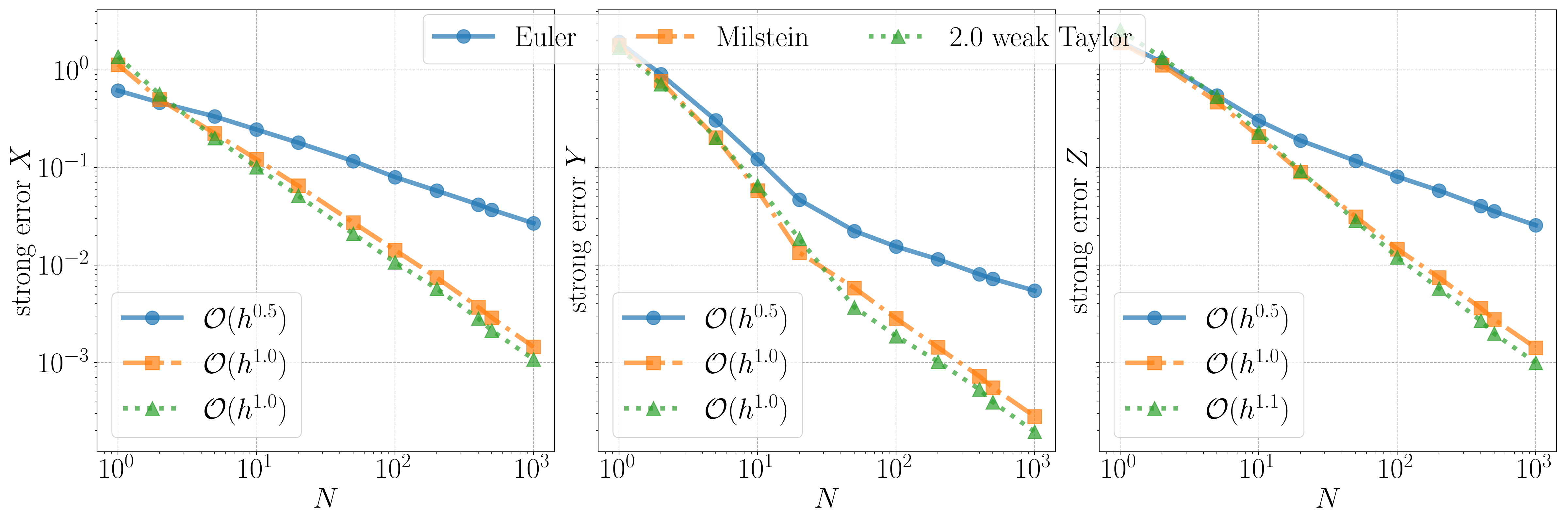}
        \caption{convergence of strong errors}
        \label{fig:example1:strong}
        \end{subfigure}
        
        \begin{subfigure}[t]{0.49\textwidth}
        \includegraphics[width=\linewidth]{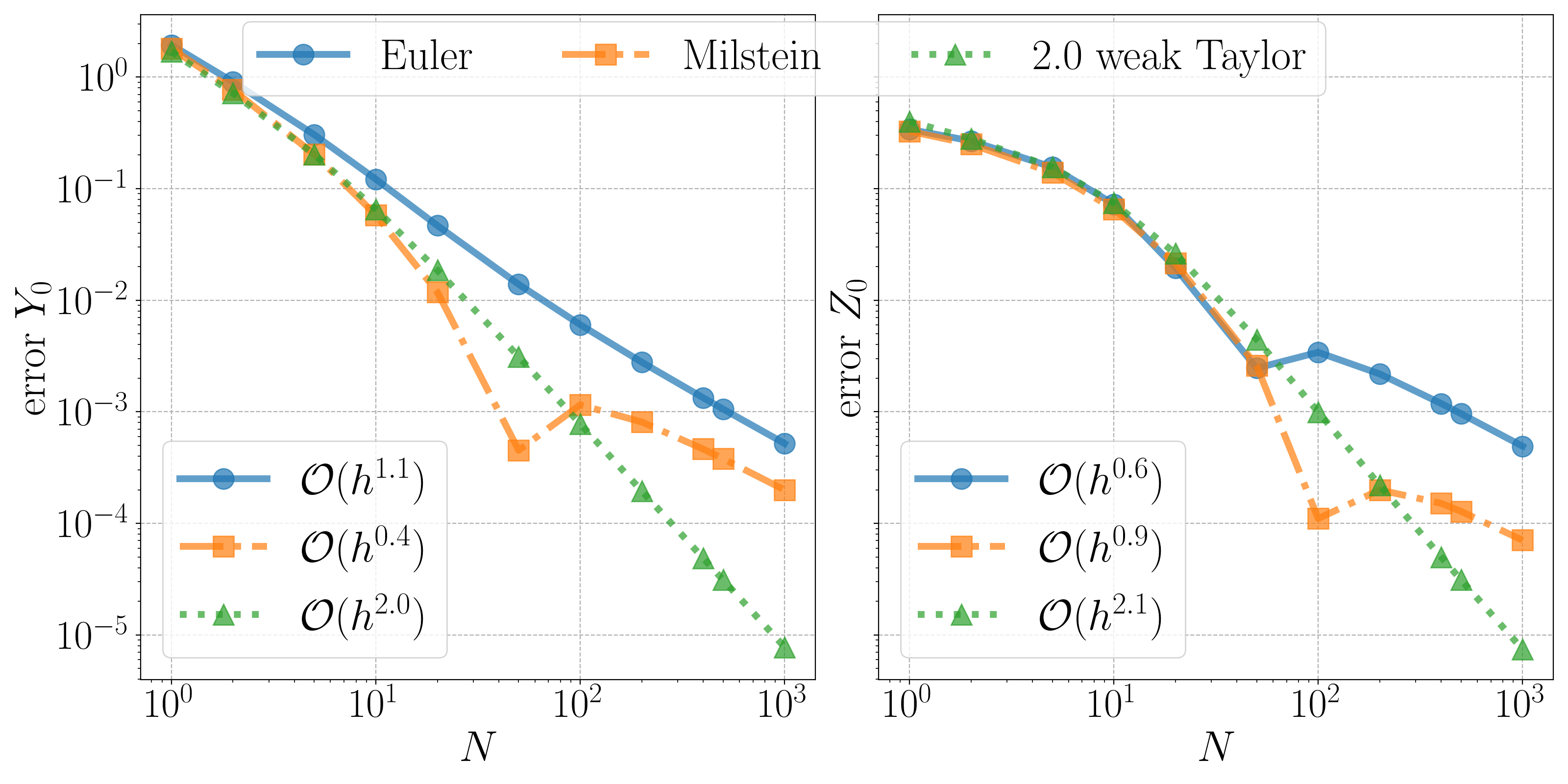}
        \caption{convergence at $t_0$}
        \label{fig:example1:weak}
        \end{subfigure}
        \begin{subfigure}[t]{\sizeonebyone\textwidth}
        \includegraphics[width=\linewidth]{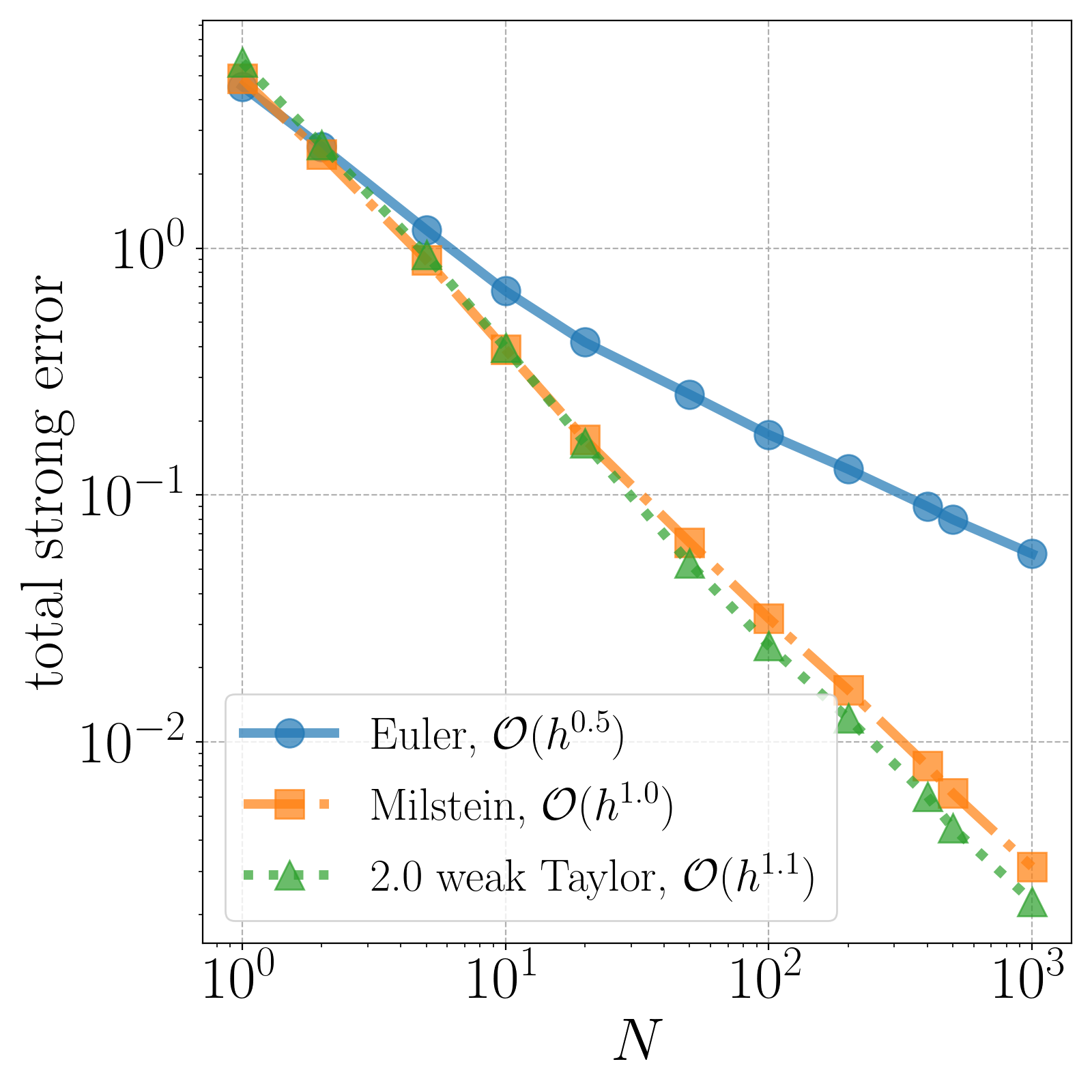}
        \caption{convergence of total strong errors}
        \label{fig:example1:strong:total}
        \end{subfigure}\hspace{0.0\textwidth}
        \begin{subfigure}[t]{\sizeonebyone\textwidth}
        \includegraphics[width=\linewidth]{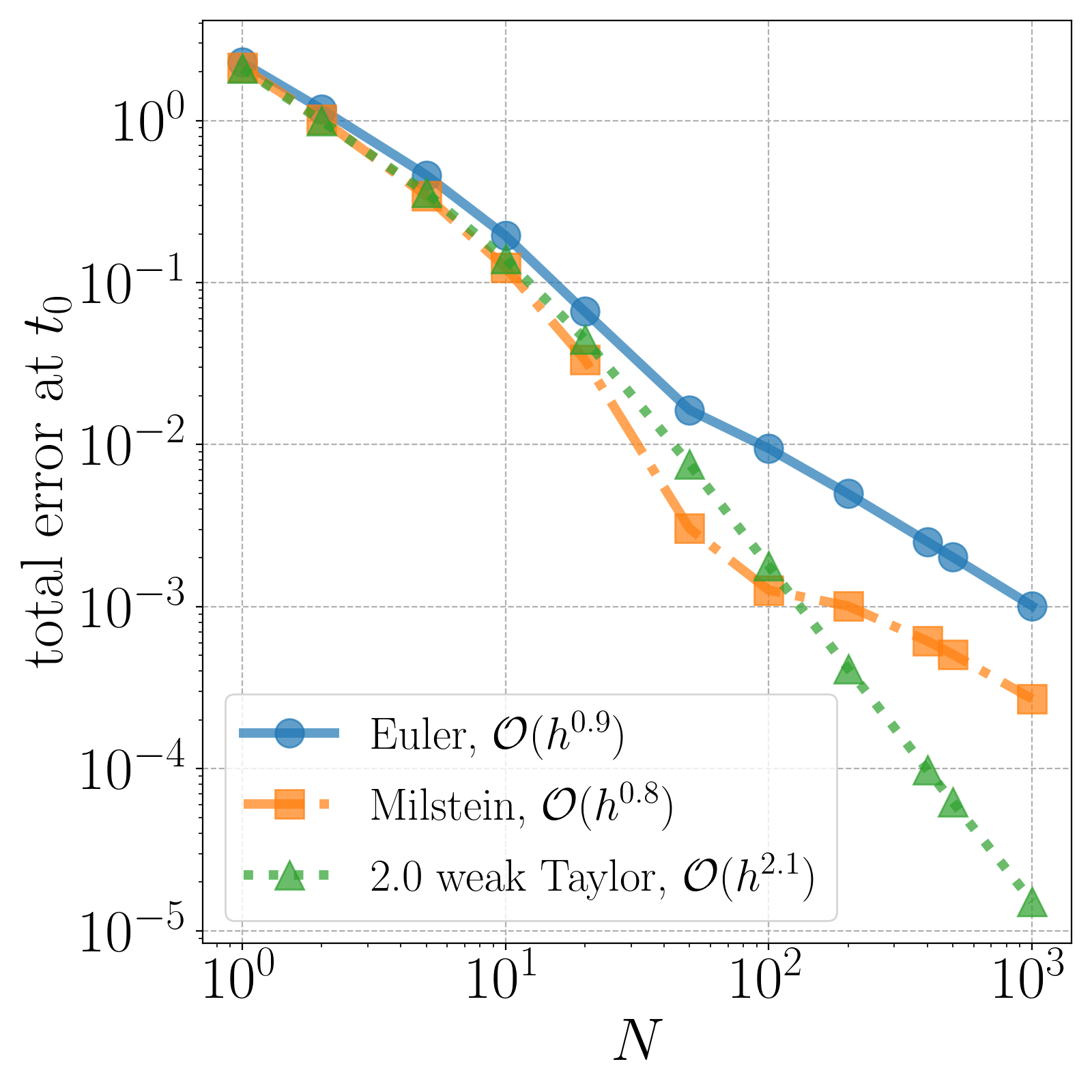}
        \caption{convergence of total errors at $t_0$}
        \label{fig:example1:weak:total}
        \end{subfigure}

        \caption{Example 1 in \eqref{eq:example1}. ($\theta_4=0, \theta_1=\theta_2=\theta_3=1/2$, $K=2^9$.)}
        \label{fig:example1}
    \end{figure}
    In order to demonstrate the slight differences between the discretizations \eqref{eq:zhao_scheme} and the one applied in \cite{ruijter_numerical_2016}, on top of the generalization provided by lemma \ref{lemma:j_k}, our first example corresponds to the special case of decoupled FBSDEs. This example appears in \cite[example 1]{ruijter_numerical_2016} and is originally from \cite{milstein_numerical_2006}. The coefficients in \eqref{eq:fbsde} read as follows
    \begin{align}\label{eq:example1}
    \begin{split}
        \mu(t, x, y, z)&=x(1+x^2)/(2+x^2),\quad \sigma(t, x, y, z)=(1+x^2)/(2+x^2),\quad g(x)=\exp(-x^2/(T+1)),\\
        f(t, x, y, z) &= \begin{aligned}[t]
            &\frac{1}{t+1}\exp(-x^2/(t+1))\left[4x^2\frac{1+x^2}{(2+x^2)^3} + \left(\frac{1+x^2}{2+x^2}\right)^2\left(1-\frac{2x^2}{t+1}-\frac{x^2}{t+1}\right)\right]\\
            &+\frac{zx}{(2+x^2)^2}\sqrt{\frac{1+y^2+\exp(-2x^2/(t+1))}{1+2y^2}}.
        \end{aligned}
    \end{split}
    \end{align}
    The analytical solution pair to \eqref{eq:fbsde:bsde} is given by the following deterministic mappings in \eqref{eq:feynman-kac}
    \begin{align}
        u(t, x)&= \exp(-x^2/(t+1)),\quad v(t, x)=-\frac{2x(1+x^2)}{(t+1)(2+x^2)}\exp(-x^2/(t+1)).
    \end{align}
    In line with \cite{ruijter_numerical_2016}, we choose $x_0=1$, $T=10$ and
    for the COS method, we fix $K=512$ Fourier coefficients, and set the domain $[a, b]$ as in\footnote{resulting in $a\approx-19.341110327048455$ and $b\approx 22.822591808529936$} \cite{ruijter_numerical_2016}. 
    
    The numerical results are collected in figure \ref{fig:example1} for all forward discretizations in \eqref{eq:sde:general_second_order_scheme}, using the generalized $\theta$-scheme \eqref{eq:zhao_scheme} with $\theta_1=\theta_2=\theta_3=1/2, \theta_4=0$. In figure \ref{fig:example1:strong}, the strong convergence rates of all three processes are collected for each method, computed according to \eqref{eq:strong-error-computation}. As we can see, when the forward transition in \eqref{eq:sde:general_second_order_scheme} is approximated by the Euler discretization \eqref{eq:coefficients:euler}, each process converges with a rate of $\mathcal{O}(h^{1/2})$ in the step size. However, this rate can be improved by the Milstein \eqref{eq:coefficients:milstein} and 2.0 weak Taylor schemes \eqref{eq:coefficients:2.0-weak-taylor}, which both show an asymptotic convergence rate of $\mathcal{O}(h)$. The convergence of the total strong errors is given in figure \ref{fig:example1:strong:total}, which demonstrates the improved first-order strong convergence with the Milstein and 2.0 weak Taylor discretizations. Figure \ref{fig:example1:weak} can be directly compared to fig. 5.1 in \cite{ruijter_numerical_2016}. Regardless of the slight difference in the backward theta-scheme considered therein, one can draw similar conclusions about the convergence of the approximation errors at $t_0=0$. Given the deterministic condition in \eqref{eq:fbsde:sde}, errors at $t_0$ coincide with weak errors, and the corresponding convergence rates admit weak convergence rates. This results in the Euler and Milstein discretizations exhibiting a weak convergence at $t_0$ with rate $\mathcal{O}(h)$ ; whereas the 2.0 weak Taylor scheme improves this to $\mathcal{O}(h^2)$. We remark that the latter convergence, in light of \eqref{eq:coefficients:2.0-weak-taylor}, also implies the convergence of the second-order derivatives (gammas) in \eqref{eq:derivative_approximations:2}. Experiments with other choices of $\theta_1, \theta_2, \theta_3$ confirmed the findings of \cite{ruijter_numerical_2016}, in terms of the weak convergence at $t_0$. In particular, by $N=10^3$, one gains approximately two orders of magnitude accuracy at $t_0$ by employing the 2.0 weak Taylor scheme. In many applications -- e.g. hedging, or portfolio allocation -- this is of high importance.
    Additionally, and in line with \cite{ruijter_numerical_2016}, we found that in order to have a strong convergence of rate $\mathcal{O}(h)$ for the Milstein and 2.0 weak Taylor schemes, it is necessary for the backward component to be discretized by a second-order scheme.

    \subsection{Example 2: partial coupling}
    The following coupled FBSDE system from \cite{negyesi_generalized_2024} is an adaptation of \cite{bender_time_2008}, including $Z$ coupling in the forward diffusion's drift coefficient. The coefficient functions in \eqref{eq:fbsde} read as follows
    \begin{align}\label{eq:example2}
        \begin{split}
            \mu(t, x, y, z)&=\kappa_y \bar{\sigma}y+\kappa_z z,\quad \sigma(t, x, y)=\bar{\sigma}y,\quad g(x)=\sin(x),\\
        f(t, x, y, z)&= -ry+1/2e^{-3r(T-t)}\bar{\sigma}^2\sin^3(x) - \kappa_y z - \kappa_z \bar{\sigma}e^{-3r(T-t)}\sin(x)\cos^2(x).
        \end{split}
    \end{align}
    The Markovian solution pair of the BSDE is given by the following deterministic mappings in \eqref{eq:feynman-kac} solving the corresponding quasi-linear PDE
    \begin{align}
        u(t, x) = e^{-r(T-t)}\sin(x),\quad v(t, x)=e^{-2r(T-t)}\bar{\sigma}\sin(x)\cos(x).
    \end{align}
    We take $T=1, X_0=\pi/4$, fix $r=0, \bar{\sigma}=0.4$ and $\kappa_y=10^{-1}$, which results in a numerically challenging equation with no monotonicity and strong coupling -- see \cite{bender_time_2008}. We fix a wide integration range by choosing $a=-3, b=5$.
    \begin{figure}[t]
        \centering
        \begin{subfigure}[t]{\textwidth}
        \includegraphics[width=\sizeonebyone\linewidth]{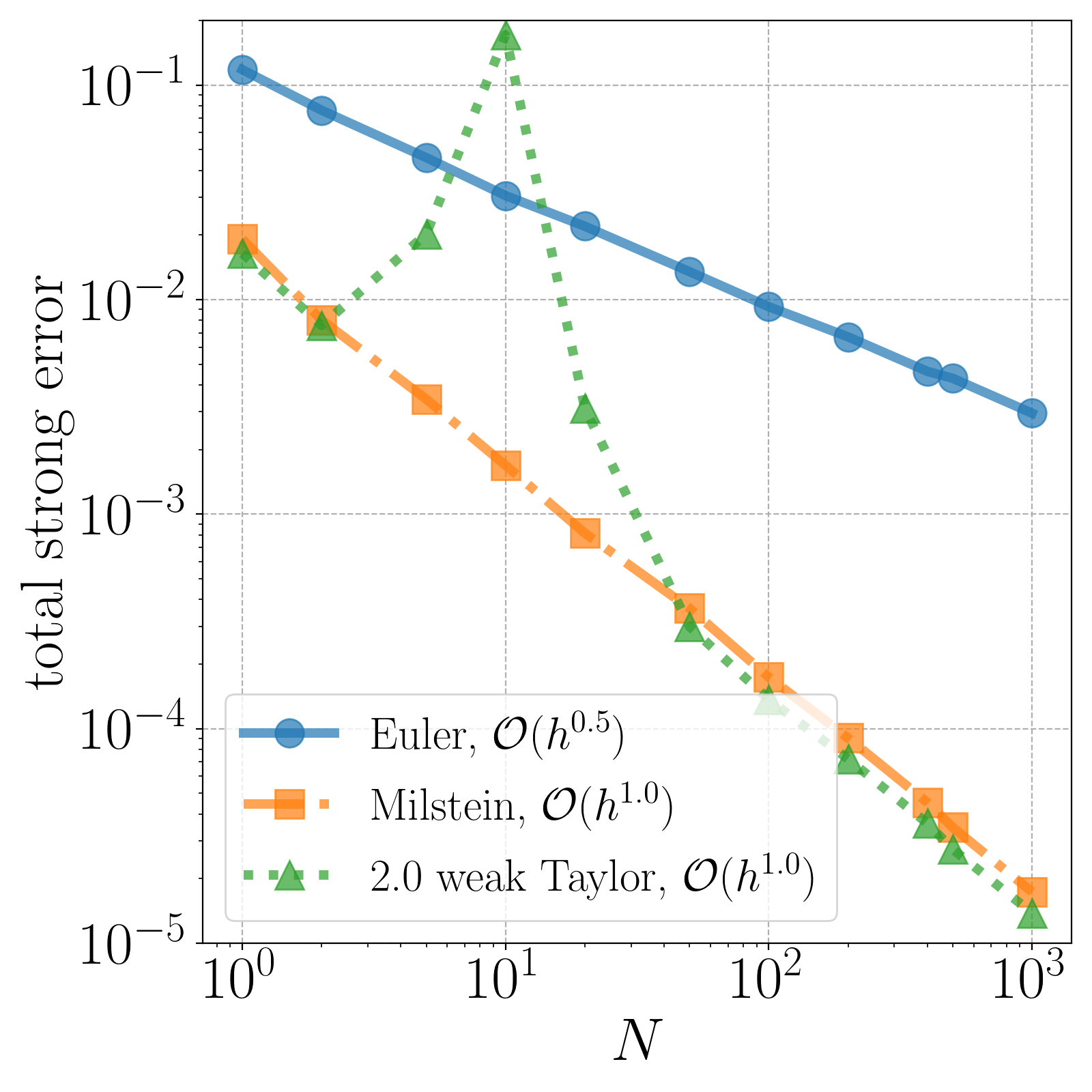}\includegraphics[width=\sizethreebyone\linewidth]{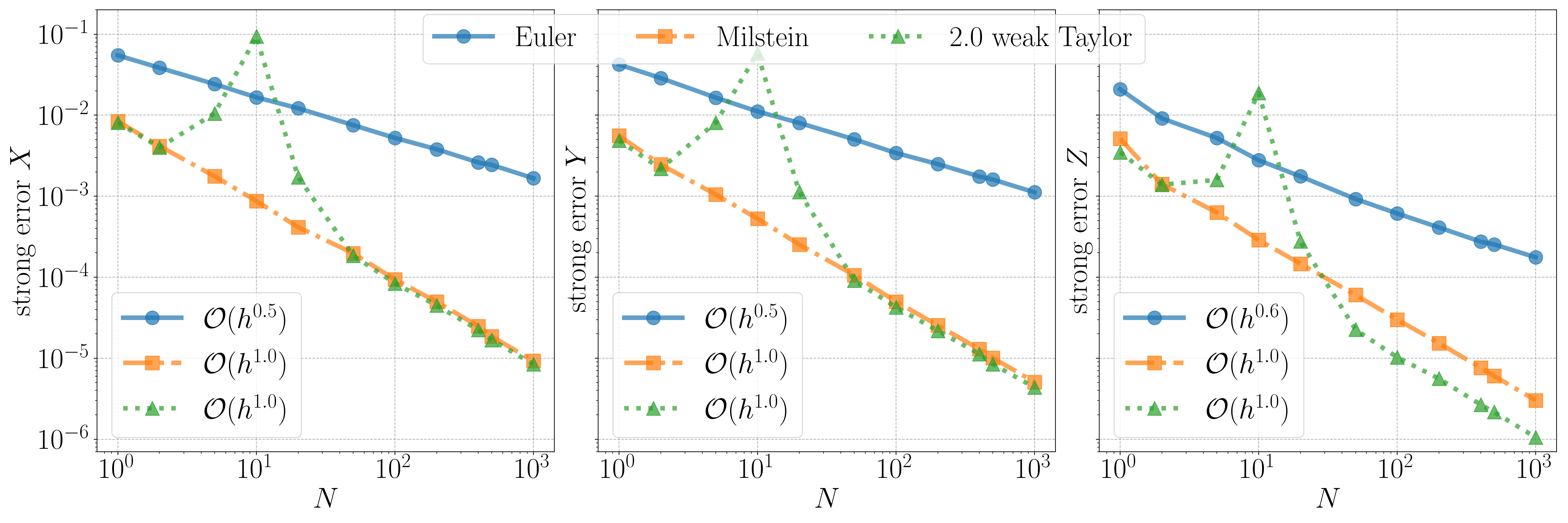}
        \caption{$\theta_4=-1/2$}
        \label{fig:example2:1:strong:zhao}
        \end{subfigure}

        \begin{subfigure}[t]{\textwidth}
        \includegraphics[width=\sizeonebyone\linewidth]{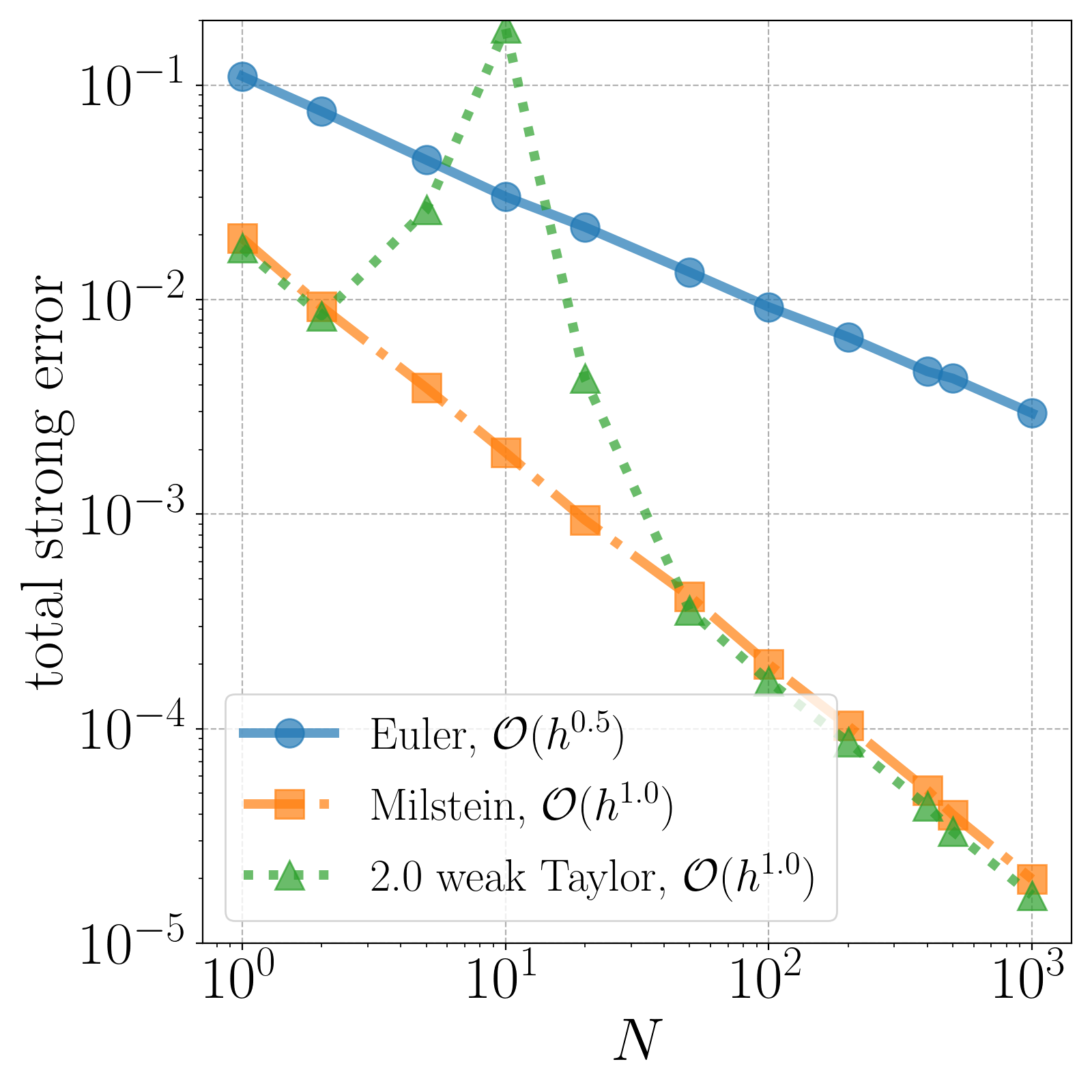}\includegraphics[width=\sizethreebyone\linewidth]{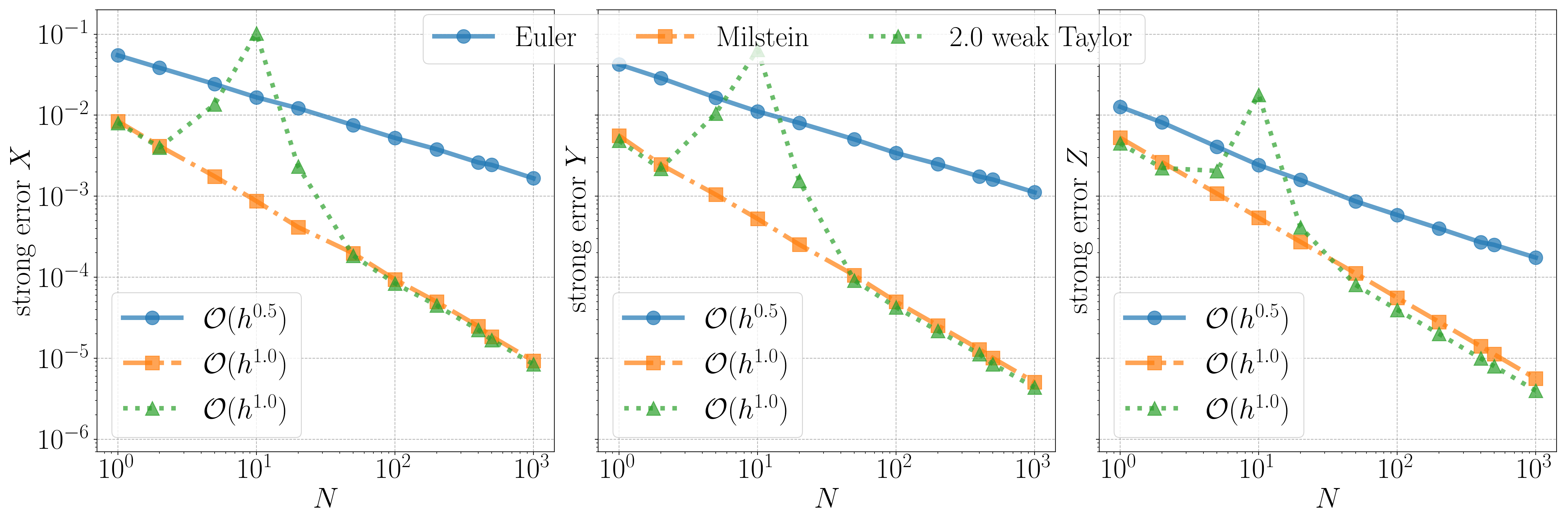}
        \caption{$\theta_4=0$}
        \label{fig:example2:1:strong:crisan}
        \end{subfigure}
        
        \caption{Example 2 in \eqref{eq:example2} with $\kappa_z=0$. Strong convergence. ($\theta_1=\theta_2=\theta_3=1/2$, $K=2^{10}$.)}
        \label{fig:example2:1:strong}
    \end{figure}
    \begin{figure}[t]
        \centering
        \begin{subfigure}[t]{\textwidth}
        \centering\includegraphics[width=\sizeonebyone\linewidth]{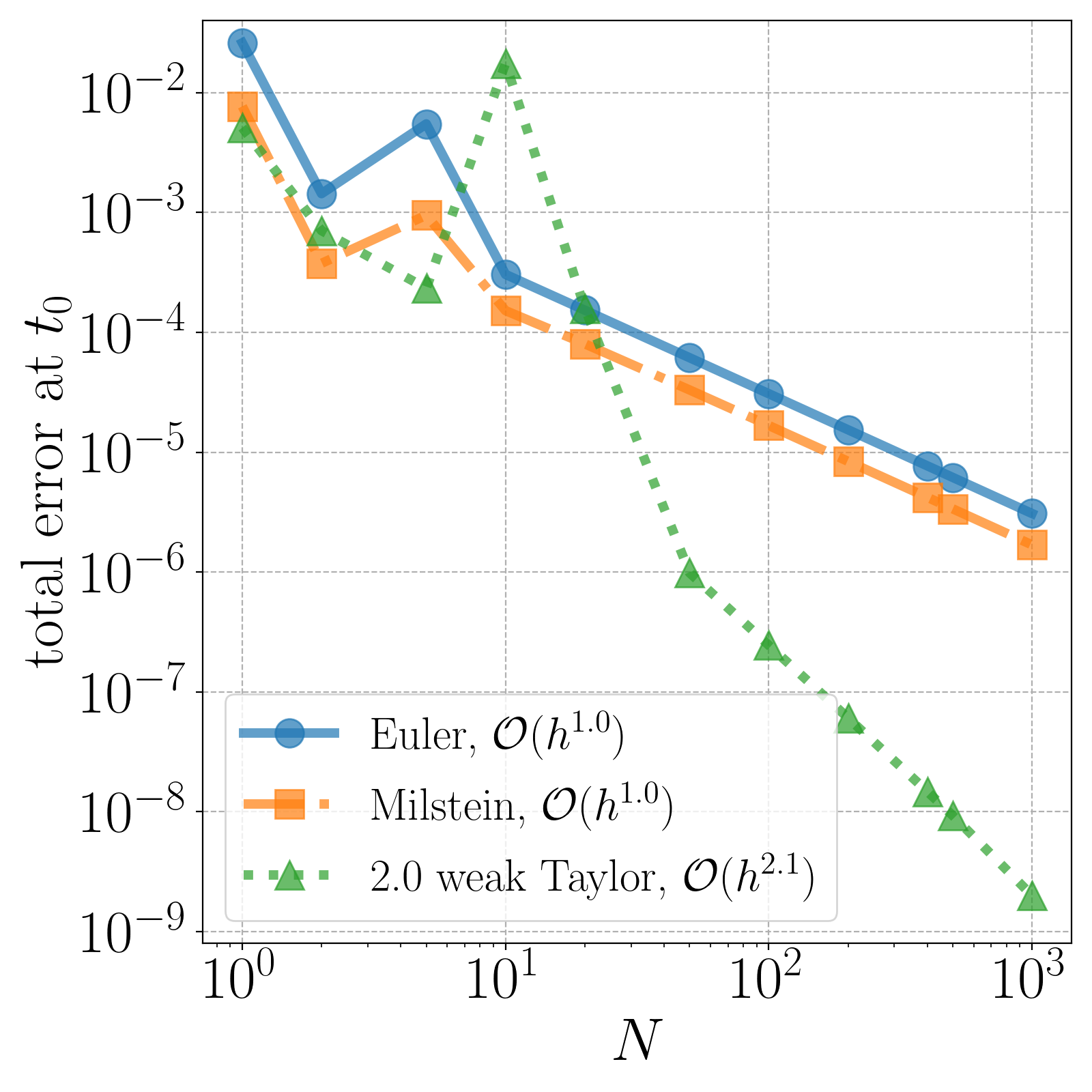}\includegraphics[width=\sizetwobyone\linewidth]{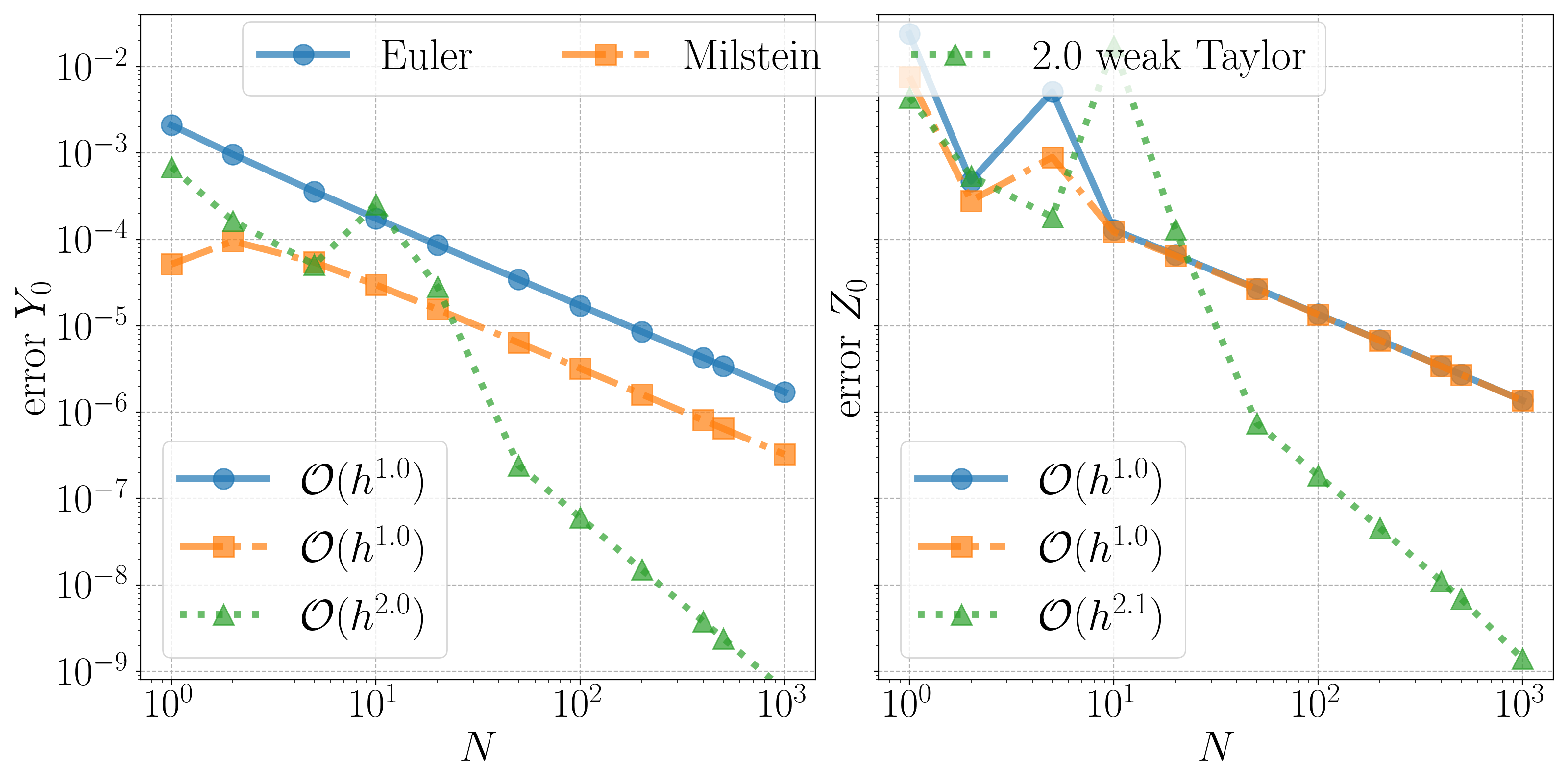}
        \caption{$\theta_4=-1/2$}
        \label{fig:example2:1:weak:zhao}
        \end{subfigure}
        
        \begin{subfigure}[t]{\textwidth}
        \centering\includegraphics[width=\sizeonebyone\linewidth]{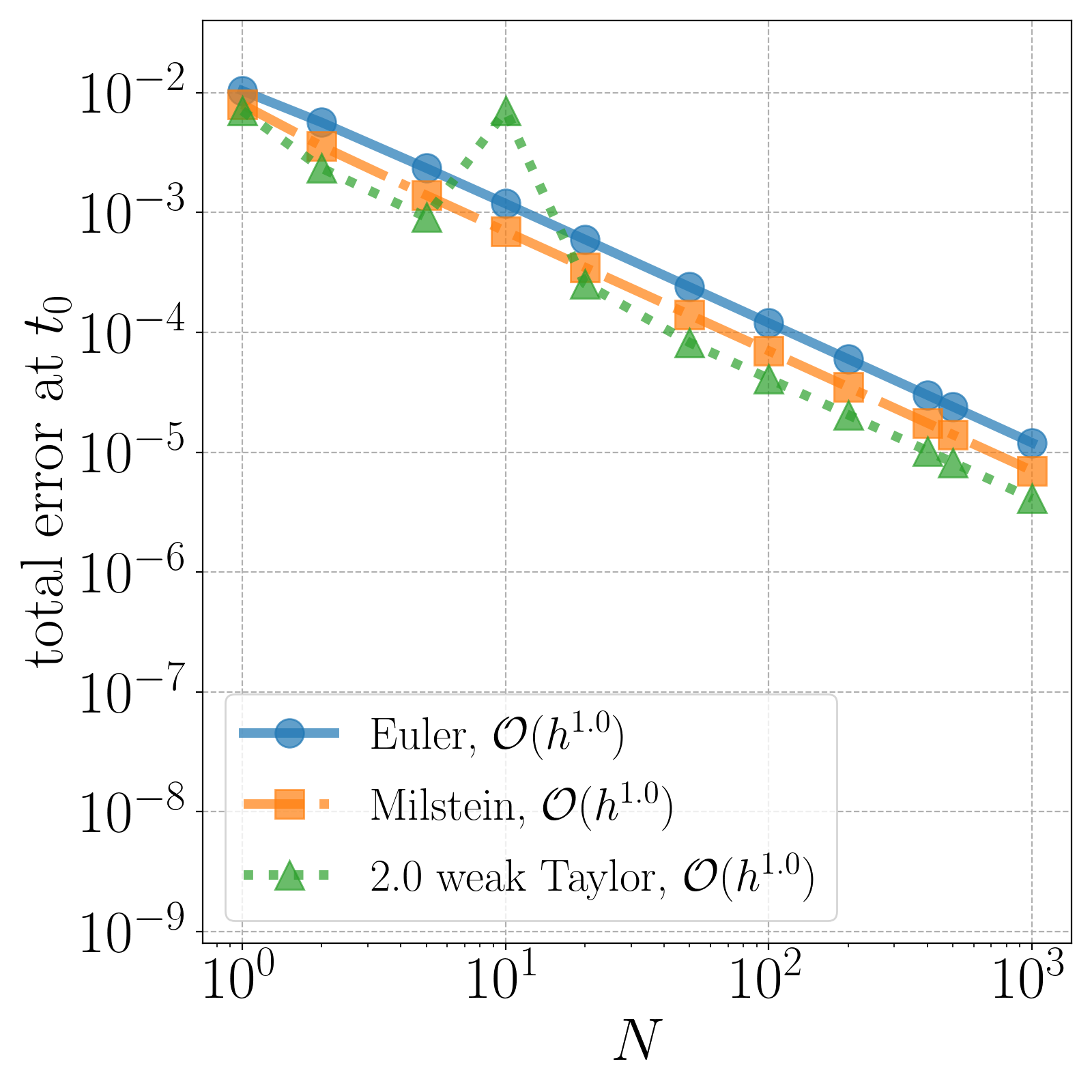}\includegraphics[width=\sizetwobyone\linewidth]{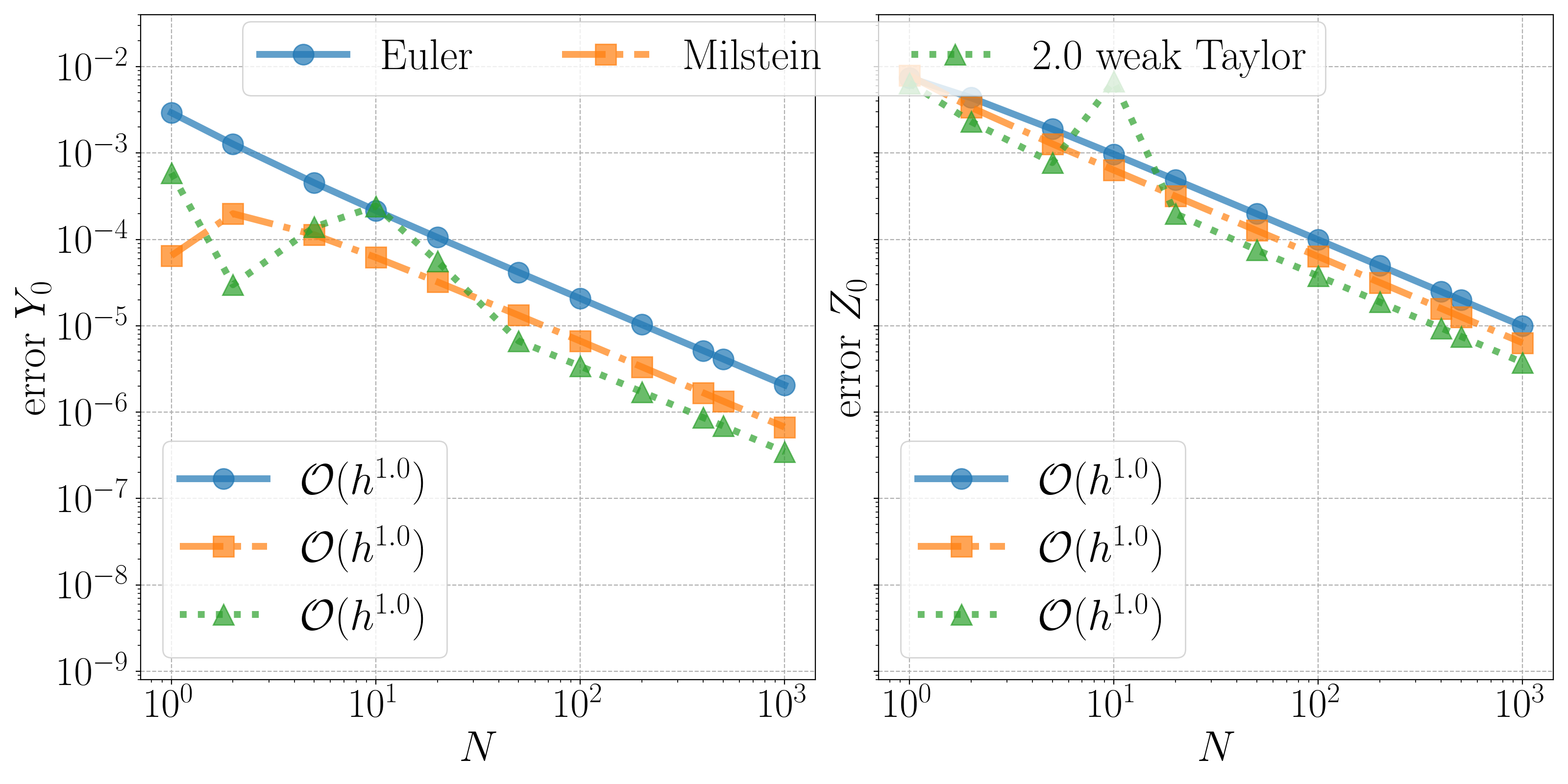}
        \caption{$\theta_4=0$}
        \label{fig:example2:1:weak:crisan}
        \end{subfigure}
        
        \caption{Example 2 in \eqref{eq:example2} with $\kappa_z=0$. Weak convergence at $t_0$. ($\theta_1=\theta_2=\theta_3=1/2$, $K=2^{10}$.)}
        \label{fig:example2:1:weak}
    \end{figure}

    \begin{figure}[t]
        \centering
        \begin{subfigure}[t]{\textwidth}
        \centering\includegraphics[width=\sizeonebyone\linewidth]{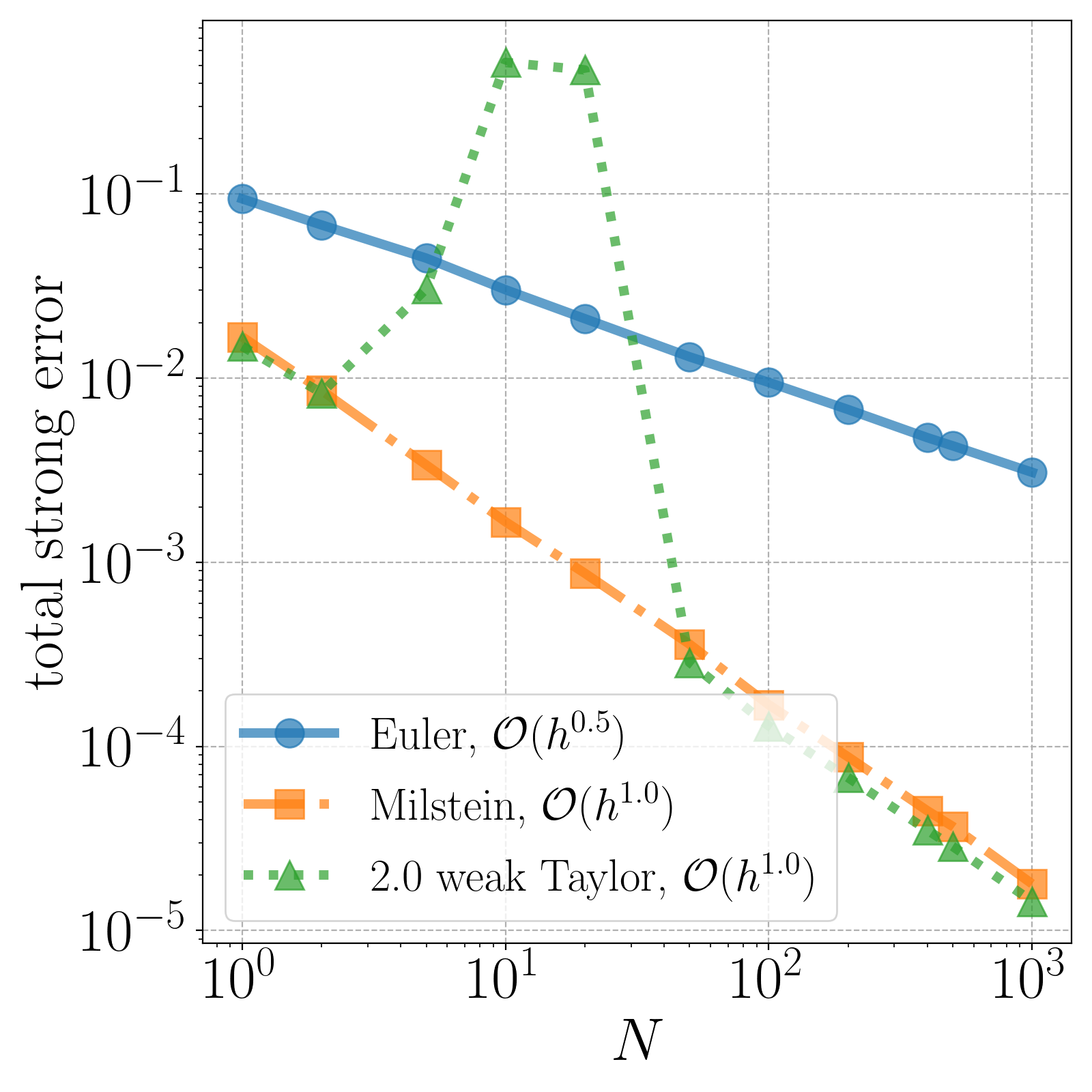}\includegraphics[width=\sizethreebyone\linewidth]{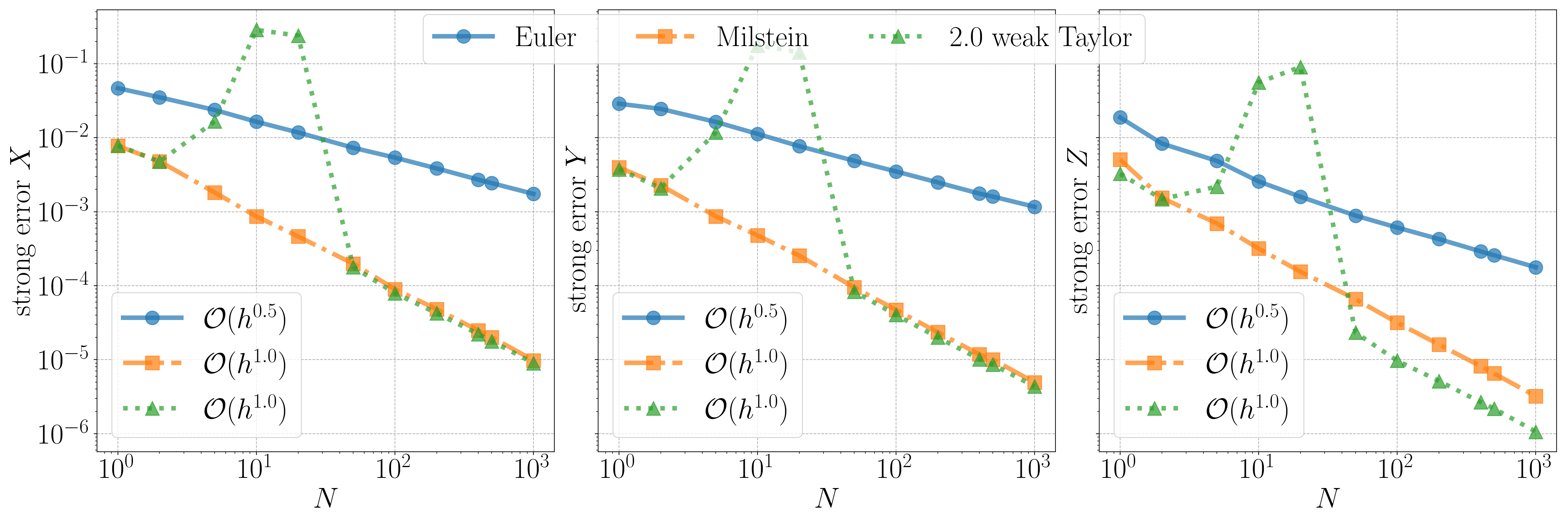}
        \caption{strong convergence}
        \label{fig:example2:2:strong}
        \end{subfigure}
        
        \begin{subfigure}[t]{\textwidth}
        \centering
        \includegraphics[width=\sizeonebyone\linewidth]{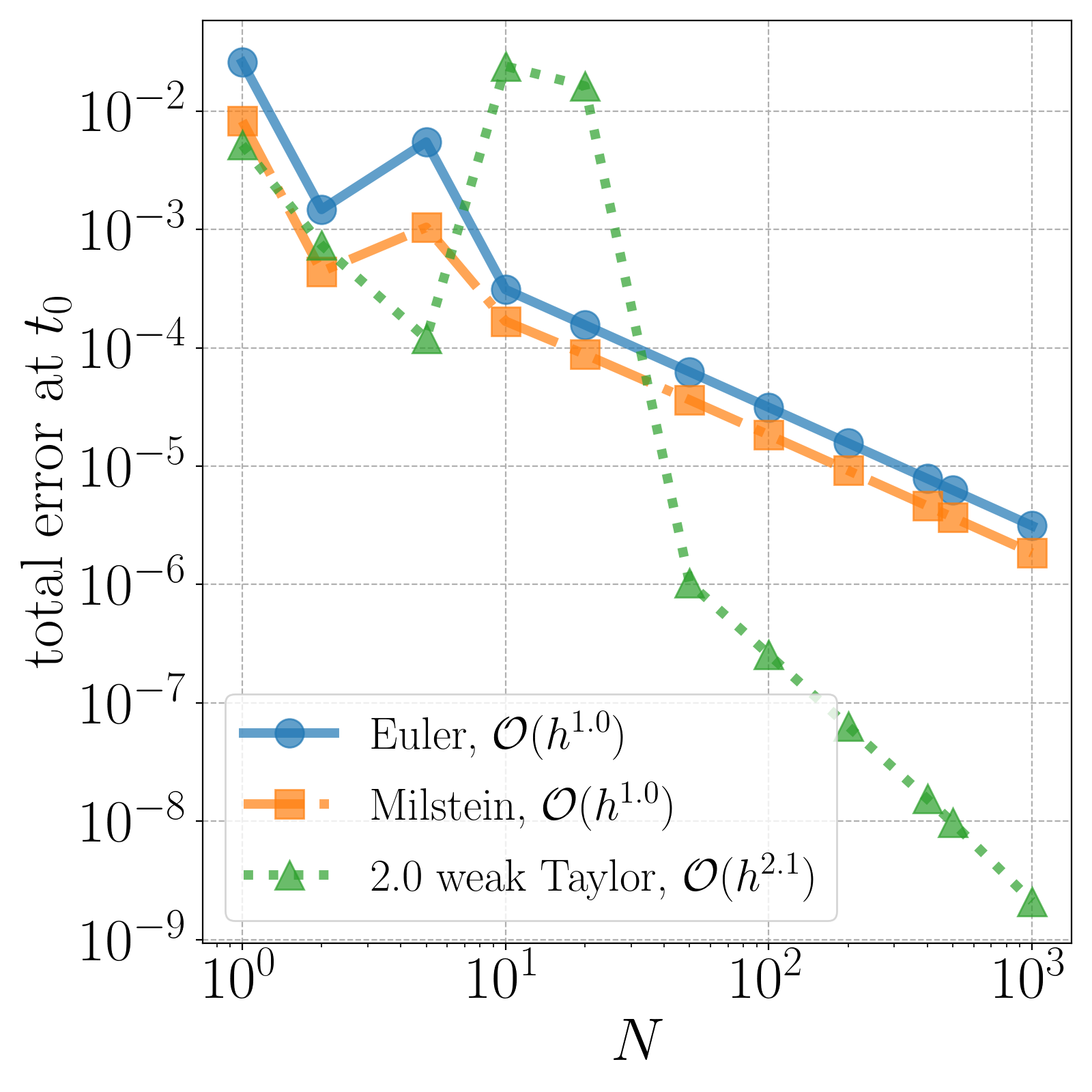}\includegraphics[width=\sizetwobyone\linewidth]{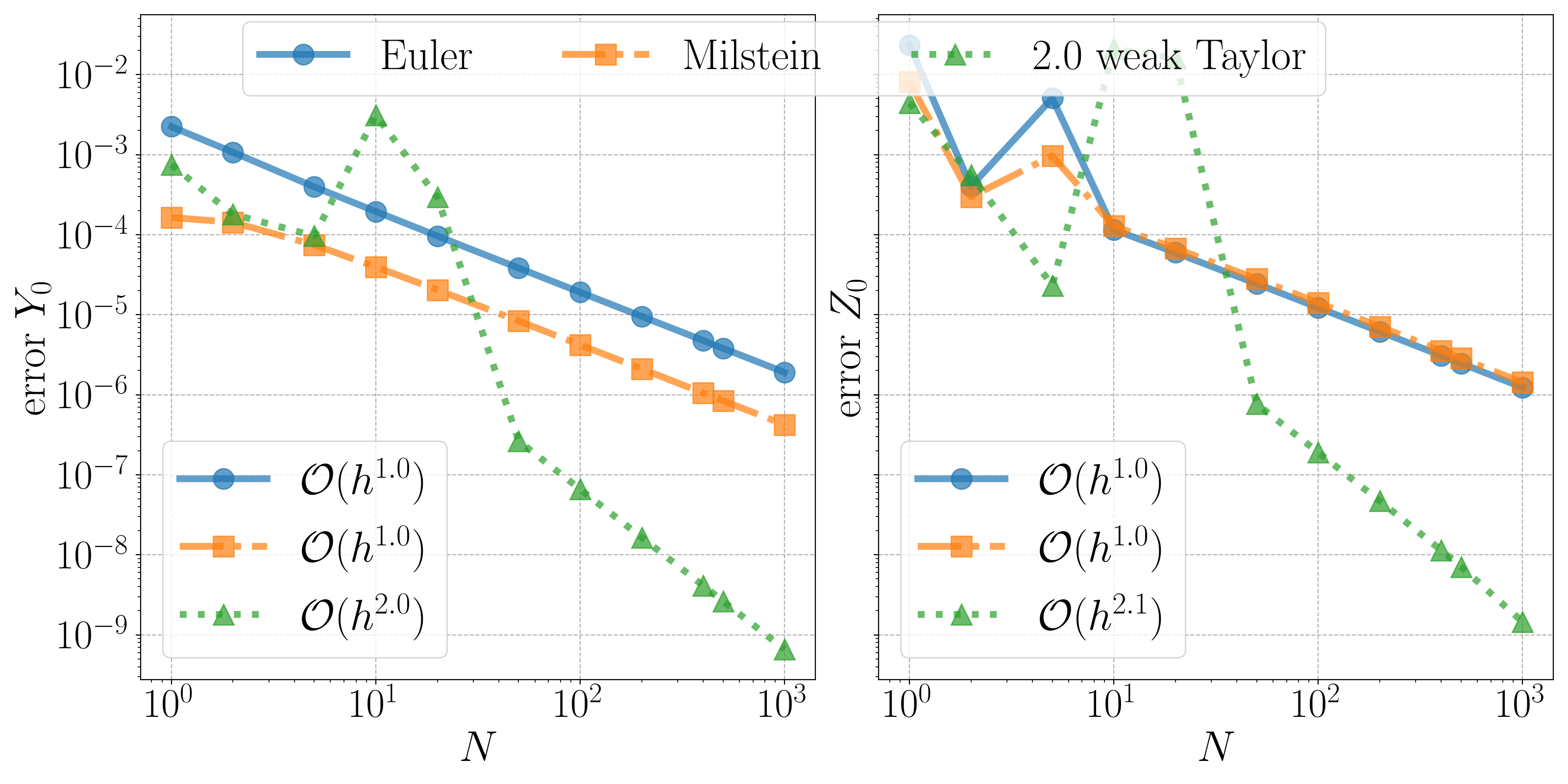}
        \caption{weak convergence at $t_0$}
        \label{fig:example2:2:weak}
        \end{subfigure}
        
        \caption{Example 2 in \eqref{eq:example2}, with $\kappa_z=10^{-2}$. ($\theta_1=\theta_2=\theta_3=1/2$, $\theta_4=-1/2$, $K=2^{10}$.)}
        \label{fig:example2:2}
    \end{figure}
    \begin{table}[t]
        \centering
        
        \begin{tabular}{c|ccccccccccc}
            $K$ & $1$ & $2$ & $2^2$ & $2^3$ & $2^4$ & $2^5$ & $2^6$ & $2^7$ & $2^8$ & $2^9$ & $2^{10}$ \\
            \hline
            Euler & $0.17$ & $0.18$ &	$0.19$ &	$0.20$ &	$0.22$ &	$0.29$ &	$0.56$ & 	$1.63$ &	$5.80$ &	$23.75$	&$107.56$\\
            Milstein & $0.19$	&$0.20$	&$0.21$	&$0.22$	&$0.24$&	$0.33$&	$0.65$&	$1.93$&	$6.89$&	$28.00$&	$126.36$\\
            2.0 weak Taylor &  $0.26$	&$0.27$	&$0.28$	&$0.29$ &	$0.31$ &	$0.40$ &	$0.72$ & 	$2.02$	&$7.02$ &	$28.40$&	$121.63$
        \end{tabular}
        \caption{Example 2 in \eqref{eq:example2}, with $\kappa_z=10^{-2}$. CPU runtime in seconds. ($\theta_1=\theta_2=\theta_3=1/2$, $\theta_4=-1/2$, $N=10^{3}$.)}
        \label{tab:example2:cpu}
    \end{table}

    For the coupling of $Z$ in the drift, we consider two different values for $\kappa_z$:
    \begin{enumerate}
        \item $\kappa_z=0$: the backward equation only couples into the forward equation through the $Y$ process;\label{example2:y_only}
        \item $\kappa_z=10^{-2}$: $Z$ is also coupled into the forward SDE but only through the drift and not the diffusion coefficient.
    \end{enumerate}

    \subsubsection[Coupling only in Y]{Coupling only in $Y$}
    We start by presenting results on the first case, when $\kappa_z=0$, i.e. the $Z$ process does not enter the forward diffusion \eqref{eq:fbsde:sde}. In order to be able to neglect the Fourier truncation error term in the second-order weak convergence of the 2.0 weak Taylor discretization for very fine time grids, we use a larger number of Fourier terms $K=2^{10}$ and remark that the rest of the results are close to identical for significantly smaller number of expansion coefficients.  The strong convergence rates for each process of the solution triple are plotted against the number of discretization points in time in figure \ref{fig:example2:1:strong} for $\theta_4=0$ (fig.\ref{fig:example2:1:strong:zhao}) and $\theta_4=-1/2$ (fig.\ref{fig:example2:1:strong:crisan}). The strong errors across the different backward discretizations are comparable, with an order of magnitude  gain when using $\theta_4=-1/2$ in the approximation of $Z$ while using Milstein or 2.0 weak Taylor forward schemes and fine time partitions. Each discretization exhibits the theoretically expected strong convergence rate predicted in table \ref{tab:convergence-rates} for both $\theta_4$ values. In fact, we recover a strong convergence rate of $\mathcal{O}(h^{1/2})$ in case of the Euler transitions, as in \cite{huijskens_efficient_2016}. However, thanks to the generalization by algorithm \ref{algorithm} to higher order schemes, we manage to improve the strong convergence rate to $\mathcal{O}(h)$ by using the Milstein and 2.0 weak Taylor transitions in \eqref{eq:sde:general_second_order_scheme}, even in the case of coupling in $Y$. In particular, due to the higher order discretizations enabled by algorithm \ref{algorithm}, the Milstein and 2.0 weak Taylor approximations achieve an almost three orders of magnitude higher accuracy in the total strong approximation error than the Euler method in \cite{huijskens_efficient_2016}, when $N=10^3$.

    The main difference between the choices in $\theta_4$ is illustrated by figure \ref{fig:example2:1:weak}, where the weak convergence rates of the approximation errors at $t_0$ are depicted. As can be seen, when the forward diffusion is approximated by either an Euler or Milstein scheme, the weak errors at $t_0$ converge with the expected $\mathcal{O}(h)$ regardless of the value of $\theta_4$. However, the same cannot be said about the 2.0 weak Taylor discretization. In fact, we find that the errors at $t_0$ only show second-order convergence when $\theta_4=-1/2$. With $\theta_4=0$, the accumulating approximation errors from the backward equation result in a slower convergence of the same $\mathcal{O}(h)$ or as with the Euler and Milstein schemes. These findings suggest that in applications where the approximation accuracy at $t_0$ is of special relevance, $\theta_4=-1/2$ may be a preferred choice for the discretization of the backward equation, when the forward transition is modeled with a 2.0 weak Taylor scheme.\footnote{The choice $\theta_4=0$ may be necessary in some applications where the solutions are highly oscillating. In such case with $\theta_4=-1/2$ one could face stability issues, whereas the second-order scheme in \eqref{eq:crisan_scheme} would preserve the convergence rate of the standard backward Euler scheme -- see \cite{crisan_second_2014}. Our results suggest that even in such case, the strong convergence rates can be improved by the second-order Taylor schemes in \eqref{eq:sde:general_second_order_scheme}.}

    \subsubsection[Z coupling in the drift]{$Z$ coupling in the drift}
    Let us consider the second case of \eqref{eq:example2}, corresponding to $\kappa_z=10^{-2}$, i.e. when the $Z$ process enters the dynamics of the forward diffusion \eqref{eq:fbsde:sde} but only through the drift.
    Similar to the previous case, we fix $K=2^{10}$, and as the conclusion on $\theta_4$ is verbatim, we only present results in the case $\theta_4=-1/2$. 

    Numerical results are collected in figure \ref{fig:example2:2}.
    The strong approximation errors are depicted in fig. \ref{fig:example2:2:strong}. As we can see, the Milstein and 2.0 weak Taylor approximation preserve their theoretically expected $\mathcal{O}(h)$ convergence rate -- improving on that of the Euler scheme -- even when $Z$ enters the forward dynamics. For fine time grids such as $N=10^3$ this results in a more than 2 orders of magnitude improvement compared to the method employed in \cite{huijskens_efficient_2016}. Moreover, similarly to the previous case, the theoretically expected weak convergence rates are recovered for each forward discretization for the approximation errors computed at $t_0$, as illustrated by fig. \ref{fig:example2:2:weak}. In fact, we find that whereas the Milstein scheme brings a marginal improvement in the accuracy at $t_0$ compared to the Euler scheme, they both exhibit first order weak convergence. On the other hand, by employing a 2.0 weak Taylor discretization in the forward diffusion's Markov transition as in \eqref{eq:coefficients:2.0-weak-taylor}, one achieves second-order convergence. This results in the approximation accuracy at $t_0$ reaching practically machine accuracy by $N=10^3$, yielding 3 orders of magnitude gain compared to Euler and Milstein.
    As discussed in section \ref{sec:computational_complexity}, this comes with additional computational complexity driven by $2\times 2$ matrix vector multiplications, in order to compute the corresponding first- and second-order derivatives of the decoupling relations in \eqref{eq:derivative_approximations:1} and \eqref{eq:derivative_approximations:2}. Table \ref{tab:example2:cpu} translates this additional computational complexity into CPU time by collecting the total runtime of the BCOS method for each forward discretization and different choices of the truncation of the Fourier series, with a fixed $N=10^3$. Unsurprisingly, the Euler approximations are the fastest as they do not require the computations of the derivatives in \eqref{eq:derivative_approximations:1} and \eqref{eq:derivative_approximations:2}. Nonetheless, the Milstein approximations are executed in merely $20\%$ additional CPU time, while gaining an extra order in strong convergence. Moreover, the 2.0 weak Taylor approximations are the most computationally expensive, as on top of the first order derivatives of the decoupling fields, they also require the computation of the second-order derivatives in \eqref{eq:derivative_approximations:2}. Interestingly, the difference between the CPU times of the Milstein and 2.0 weak Taylor approximations vanishes as $K$ increases, which is explained by the reduced number of Picard iterations in the computation of the implicit conditional expectation of $y$ in order to reach our desired tolerance level -- see discussion in section \ref{sec:computational_complexity}. We remark that the difference in computation times also shows that employing a higher-order Taylor scheme in the approximation of the forward SDE of a coupled FBSDE system is advantegeous compared to Richardson extrapolating at $t_0$ with the Euler scheme as in \cite[sec. 4.3]{huijskens_efficient_2016}. Namely, since Richardson extrapolation would require re-running the method with a finer time grid (e.g. $h/2$) once, doing so the computation time doubles, while for smoothly converging errors at $t_0$ the weak convergence rate is improved to $\mathcal{O}(h^2)$. Nonetheless, Richardson extrapolation does not improve the $\mathcal{O}(h^{1/2})$ strong convergence rate of the Euler scheme , while being both more restrictive and expensive than a 2.0 weak Taylor approximation, as indicated by table \ref{tab:example2:cpu}.

    \subsection{Example 3: fully coupled FBSDE, stochastic optimal control}
    Our final numerical example is a fully-coupled FBSDE system which is related to a linear-quadratic stochastic optimal control problem and is derived from the stochastic maximum principle -- see e.g. \cite{yong_stochastic_1999, pham_continuous-time_2009}. For the derivation of the corresponding FBSDE, we refer to \cite{huang_convergence_2024} and the references therein.
    The coefficients in \eqref{eq:fbsde} read as follows
    \begin{align}\label{eq:example3}
        \begin{split}
            \mu(t, x, y, z)&= \left(A-B\frac{R_{xu}}{{R_u}}\right)x+\frac{B^2}{R_u}y+\frac{BD}{R_u}z+\beta,\\
        \sigma(t, x, y, z)&=\left(C-D\frac{R_{xu}}{{R_u}}\right)x+\frac{DB}{R_u}y + \frac{D^2}{R_u}z+\Sigma,\quad g(x)=-Gx,\\
        f(t, x, y, z)&=\left(A-\frac{BR_{xu}}{R_u}\right)y + \left(C- \frac{DR_{xu}}{R_u}\right)z- \left(R_x - \frac{R_{xu}}{R_u}\right)x.
        \end{split}
    \end{align}
    Notice that, unlike in the previous example, the diffusion coefficient of the forward SDE also takes $Z$ as an argument.
    The semi-analytical solution can be obtained by numerical integration of a system of Ricatti ODEs, with practically arbitrary accuracy -- see \cite[eq. (45)]{huang_convergence_2024}. We take $N'=10^6$ steps in order to compute the reference solution over a refined time partition. As a one-dimensional version of example 1 in \cite{huang_convergence_2024}, we consider the parameter values $A=-1, B=0.1, \beta=0, C=1, D=0.01, \Sigma=0.05, R_x=2, R_{xu}=0, R_u=2, G=2.$ The truncation range in the BCOS approximations is set to $a=-5, b=5$.

    The convergence results are depicted in figure \ref{fig:example3}. In particular, as can be seen in fig. \ref{fig:example3:strong}, each hereby considered forward discretization exhibits its theoretically expected strong convergence rate. In case of the Euler scheme, all processes convergence with a rate of $\mathcal{O}(h^{1/2})$, whereas for the Milstein and 2.0 weak Taylor approximations, this rate is improved to $\mathcal{O}(h)$. Even though the strong convergence rate with the latter two schemes is first-order in both cases, the 2.0 weak Taylor approximations admit an advantageous constant, resulting in an order of magnitude higher overall strong approximation accuracy. 
    The weak convergence errors are collected in fig. \ref{fig:example3:weak}, from which we can draw similar conclusions as for our earlier examples. All forward discretizations preserve their theoretical weak convergence rates to the fully-coupled FBSDE setting of \eqref{eq:example3}. In fact, the Euler and Milstein approximations agree to errors converging with a rate of $\mathcal{O}(h)$ at $t_0$, to which the 2.0 weak Taylor approximation brings a significant improvement by speeding up convergence at $t_0$ to second-order in exchange for a marginally higher total CPU time -- runtimes are comparable to table \ref{tab:example2:cpu} and are thus omitted.

    As equation \eqref{eq:example3} is derived from a stochastic optimal control problem, see \cite{huang_convergence_2024} and the references therein, the approximation accuracy for the solution of the forward SDE is of special importance. As we can see, the methods proposed in the present paper using Milstein and 2.0 weak Taylor approximations for the numerical resolution of the forward diffusion do not only result in a higher order strong convergence rate, but also significantly improve the approximation accuracy for a given $N$ in the controlled forward diffusion. According to fig. \ref{fig:example3:strong}, this results in $2$ orders of magnitude gain in strong approximation accuracy; and almost $4$ orders of magnitude improvement in the approximation quality at $t_0$, when $N=10^3$ compared to the Euler method deployed in \cite{huijskens_efficient_2016}. 

    Finally, in order to assess the influence of the truncation in the Fourier cosine expansions, we collect strong and weak approximations errors in tables \ref{tab:example3:strong} and \ref{tab:example3:weak}, respectively, for different values of $K$ and $N$. Comparing tables \ref{tab:example3:strong:euler} with \ref{tab:example3:strong:milstein}, we find that both the Euler and Milstein discretizations are robust with respect to the number of Fourier coefficients $K$. In particular, from $K=2^7$ the main source of strong approximation errors is the time discretization term. In line with fig. \ref{fig:example3:strong}, the Milstein approximation proposed in this paper yields $2$ orders of magnitude improvement to the total strong approximation error compared to the Euler scheme of \cite{huijskens_efficient_2016}, independently of the choice of $K$.

    Similar conclusions can be drawn from the comparison of tables \ref{tab:example3:weak:milstein} and \ref{tab:example3:weak:2.0-weak-taylor}, which collect the weak approximation errors of the Milstein and 2.0 weak Taylor schemes for the fully-coupled equation \eqref{eq:example3}. In case of the Milstein scheme, the BCOS method is not sensitive to the choice of $K$ and the errors at $t_0$ are dominated by the time discretization. On the other hand, as can be seen from tab. \ref{tab:example3:weak:2.0-weak-taylor}, in case of the 2.0 weak Taylor discretization, due to the higher order convergence, the BCOS method more quickly reaches an error level, where the Fourier truncation becomes prominent. In particular, for the errors in $Z_0$, second-order convergence is not fully reached with $K=2^7$ Fourier terms only, and one can gain an additional order of accuracy by choosing a $K=2^{10}$ large expansion instead.

    Summarizing the implications of tables \ref{tab:example3:strong} and \ref{tab:example3:weak}, the BCOS method in algorithm \ref{algorithm} is robust with respect to the number of Fourier terms $K$ in the cosine expansions, even in the case of fully-coupled FBSDEs. Strong approximations quickly converge to the time discretization errors in $K$. In terms of weak approximation errors at $t_0$ using the 2.0 weak Taylor scheme in the forward SDE, one may need to enlarge the truncated Fourier series in order to preserve second order convergence, as the corresponding time discretization error decays with a faster rate of $\mathcal{O}(h^2)$.

    \begin{table}[t]
        \centering
        \begin{subtable}{\textwidth}
            \centering
            \begin{tabular}{l|ccc|ccc|ccc}
           &  \multicolumn{3}{c}{strong error $X$} & \multicolumn{3}{c}{strong error $Y$} & \multicolumn{3}{c}{strong error $Z$}\\
           $N\setminus K$  & $128$ & $512$ & $1024$ & $128$ & $512$ & $1024$ & $128$ & $512$ & $1024$\\ 
           \hline
           $10$ & $\num{7.7E-03}$ & $\num{7.7E-03}$ & $\num{7.7E-03}$ & $\num{1.5E-02}$ & $\num{1.5E-02}$ & $\num{1.5E-02}$ & $\num{3.3E-03}$ & $\num{3.3E-03}$ & $\num{3.3E-03}$\\
           $100$ & $\num{2.4E-03}$ & $\num{2.4E-03}$ & $\num{2.4E-03}$ & $\num{4.7E-03}$ & $\num{4.7E-03}$ & $\num{4.7E-03}$ & $\num{8.3E-04}$ & $\num{8.3E-04}$ & $\num{8.3E-04}$\\
           $400$ & $\num{1.2E-03}$ & $\num{1.2E-03}$ & $\num{1.2E-03}$ & $\num{2.4E-03}$ & $\num{2.4E-03}$ & $\num{2.4E-03}$ & $\num{4.1E-04}$ & $\num{4.1E-04}$ & $\num{4.1E-04}$\\
           $1000$ & $\num{7.7E-04}$ & $\num{7.7E-04}$ & $\num{7.7E-04}$ & $\num{1.5E-03}$ & $\num{1.5E-03}$ & $\num{1.5E-03}$ & $\num{2.7E-04}$ & $\num{2.7E-04}$ & $\num{2.7E-04}$
        \end{tabular}
        \caption{Euler with \eqref{eq:coefficients:euler}}
        \label{tab:example3:strong:euler}
        \end{subtable}

        \begin{subtable}{\textwidth}
            \centering
            \begin{tabular}{l|ccc|ccc|ccc}
           &  \multicolumn{3}{c}{strong error $X$} & \multicolumn{3}{c}{strong error $Y$} & \multicolumn{3}{c}{strong error $Z$}\\\
           $N\setminus K$  & $128$ & $512$ & $1024$ & $128$ & $512$ & $1024$ & $128$ & $512$ & $1024$\\ 
           \hline
           $10$ & $\num{1.7E-03}$ & $\num{1.7E-03}$ & $\num{1.7E-03}$ & $\num{3.5e-3}$ & $\num{3.5e-3}$ & $\num{3.5e-3}$ & $\num{1.9E-03}$ & $\num{1.9E-03}$ & $\num{1.9E-03}$\\
           $100$ & $\num{1.7E-4}$ & $\num{1.7E-4}$ & $\num{1.7E-4}$ & $\num{3.3E-4}$ & $\num{3.3E-4}$ & $\num{3.3E-4}$ & $\num{1.9E-04}$ & $\num{1.9E-04}$ & $\num{1.9E-04}$\\
           $400$ & $\num{4.3E-05}$ & $\num{4.3E-05}$ & $\num{4.3E-05}$ & $\num{8.7E-05}$ & $\num{8.7E-05}$ & $\num{8.7E-05}$ & $\num{4.9E-05}$ & $\num{4.9E-05}$ & $\num{4.9E-05}$\\
           $1000$ & $\num{1.7E-05}$ & $\num{1.7E-05}$ & $\num{1.7E-05}$ & $\num{3.4E-05}$ & $\num{3.4E-05}$ & $\num{3.4E-05}$ & $\num{2.0E-05}$ & $\num{2.0E-05}$ & $\num{2.0E-05}$
        \end{tabular}
        \caption{Milstein with \eqref{eq:coefficients:milstein}}
        \label{tab:example3:strong:milstein}
        \end{subtable}

        \caption{Example 3 in \eqref{eq:example3}. Strong approximation errors with various forward discretizations in \eqref{eq:sde:general_second_order_scheme}, and for different values of $K$ and $N$. ($\theta_1=\theta_2=\theta_3=1/2$, $\theta_4=-1/2$, $K=2^{10}$.)}
        \label{tab:example3:strong}
    \end{table}

    \begin{table}[t]
        \centering
        
        \begin{subtable}{\textwidth}
            \centering
            \begin{tabular}{l|ccc|ccc}
           &  \multicolumn{3}{c}{error $Y_0$} & \multicolumn{3}{c}{error $Z_0$}\\
           $N\setminus K$  & $128$ & $512$ & $1024$ & $128$ & $512$ & $1024$\\ 
           \hline
           $10$ & $\num{8.8E-4}$ & $\num{8.8E-4}$ & $\num{8.8E-4}$ & $\num{3.4E-3}$ & $\num{3.4E-3}$ & $\num{3.4E-3}$\\
           $100$ & $\num{8.6E-5}$ & $\num{8.6E-5}$ & $\num{8.6E-5}$ & $\num{3.2E-4}$ & $\num{3.2E-4}$ & $\num{3.2E-4}$\\
           $400$ & $\num{2.2E-5}$ & $\num{2.2E-5}$ & $\num{2.2E-5}$ & $\num{7.4E-5}$ & $\num{7.4E-5}$ & $\num{7.4E-5}$\\
           $1000$ & $\num{8.6E-6}$ & $\num{8.6E-6}$ & $\num{8.6E-6}$ & $\num{2.4E-5}$ & $\num{2.4E-5}$ & $\num{2.4E-5}$
        \end{tabular}
        \caption{Milstein with \eqref{eq:coefficients:milstein}}
        \label{tab:example3:weak:milstein}
        \end{subtable}

        \begin{subtable}{\textwidth}
            \centering
            \begin{tabular}{l|ccc|ccc}
           &  \multicolumn{3}{c}{error $Y_0$} & \multicolumn{3}{c}{error $Z_0$}\\
           $N\setminus K$  & $128$ & $512$ & $1024$ & $128$ & $512$ & $1024$\\ 
           \hline
           $10$ & $\num{1.5E-5}$ & $\num{1.5E-5}$ & $\num{2.1E-5}$ & $\num{7.1E-5}$ & $\num{6.9E-5}$ & $\num{2.1E-4}$\\
           $100$ & $\num{8.4E-8}$ & $\num{8.4E-8}$ & $\num{8.4E-8}$ & $\num{1.0E-6}$ & $\num{7.2E-7}$ & $\num{7.1E-7}$\\
           $400$ & $\num{5.2E-8}$ & $\num{5.2E-8}$ & $\num{5.2E-8}$ & $\num{3.2E-7}$ & $\num{6.0E-8}$ & $\num{5.1E-8}$\\
           $1000$ & $\num{2.2E-8}$ & $\num{2.2E-8}$ & $\num{2.2E-8}$ & $\num{2.8E-7}$ & $\num{2.0E-8}$ & $\num{1.1E-8}$
        \end{tabular}
        \caption{2.0 weak Taylor with \eqref{eq:coefficients:2.0-weak-taylor}}
        \label{tab:example3:weak:2.0-weak-taylor}
        \end{subtable}

        \caption{Example 3 in \eqref{eq:example3}. Weak approximation errors at $t_0$, with various forward discretizations in \eqref{eq:sde:general_second_order_scheme}, and for different values of $K$ and $N$. ($\theta_1=\theta_2=\theta_3=1/2$, $\theta_4=-1/2$, $K=2^{10}$.)}
        \label{tab:example3:weak}
    \end{table}

    \begin{figure}[t]
        \centering
        \begin{subfigure}[t]{\textwidth}
        \centering\includegraphics[width=\sizeonebyone\linewidth]{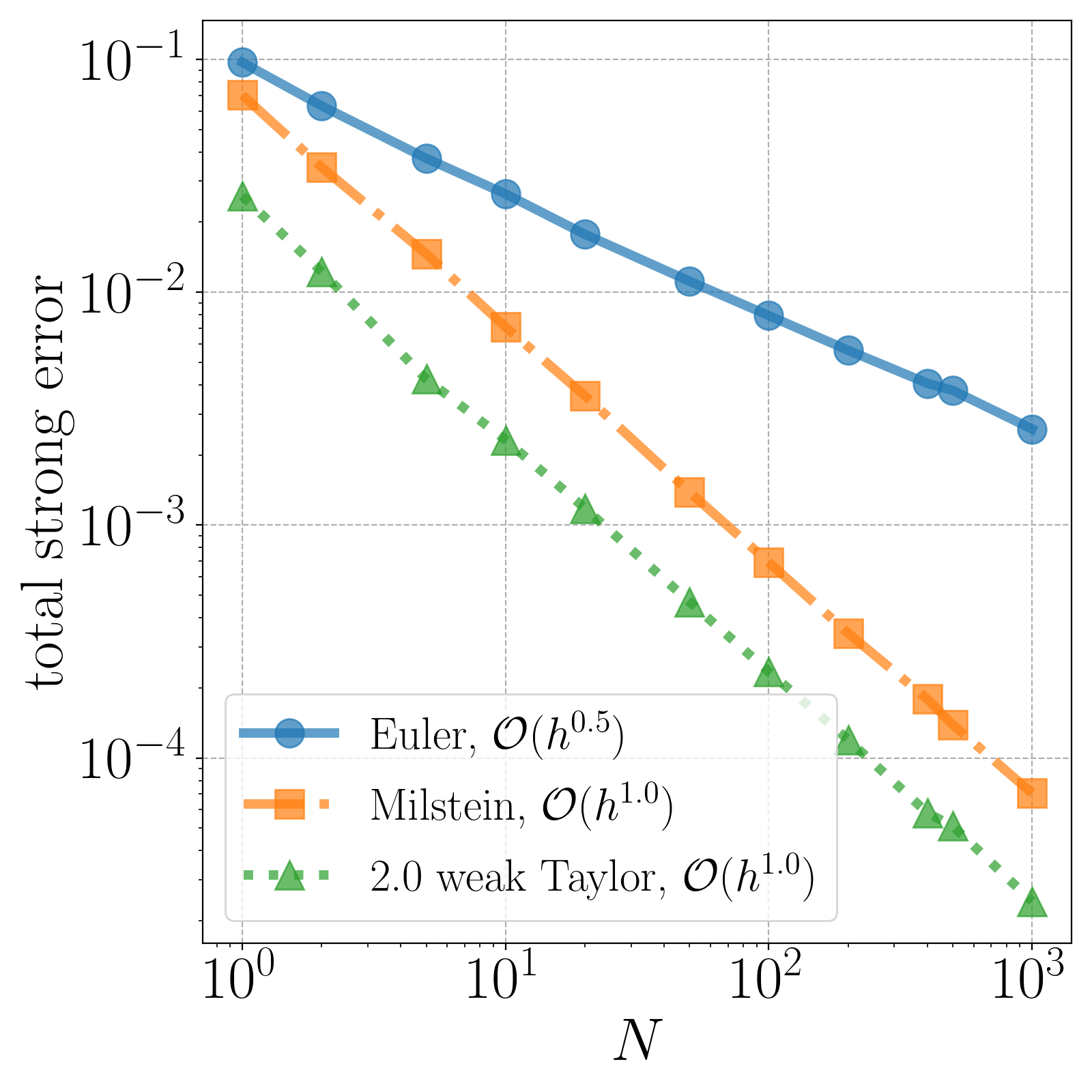}\includegraphics[width=\sizethreebyone\linewidth]{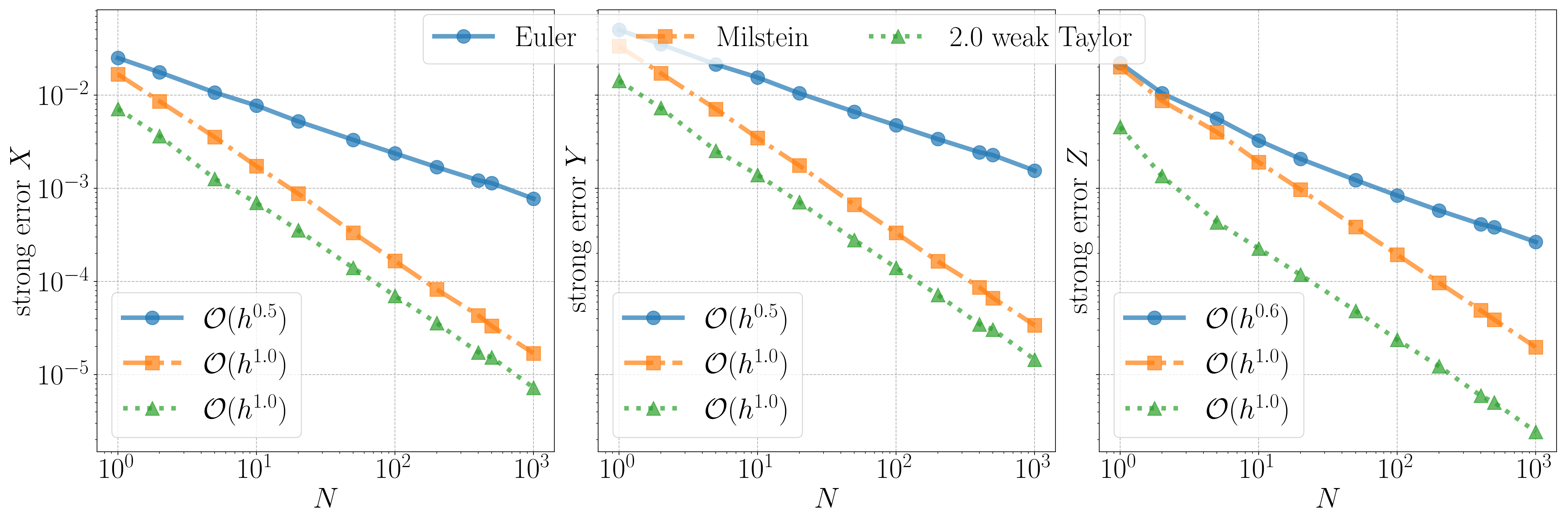}
        \caption{strong convergence}
        \label{fig:example3:strong}
        \end{subfigure}
        
        \begin{subfigure}[t]{\textwidth}
        \centering
        \includegraphics[width=\sizeonebyone\linewidth]{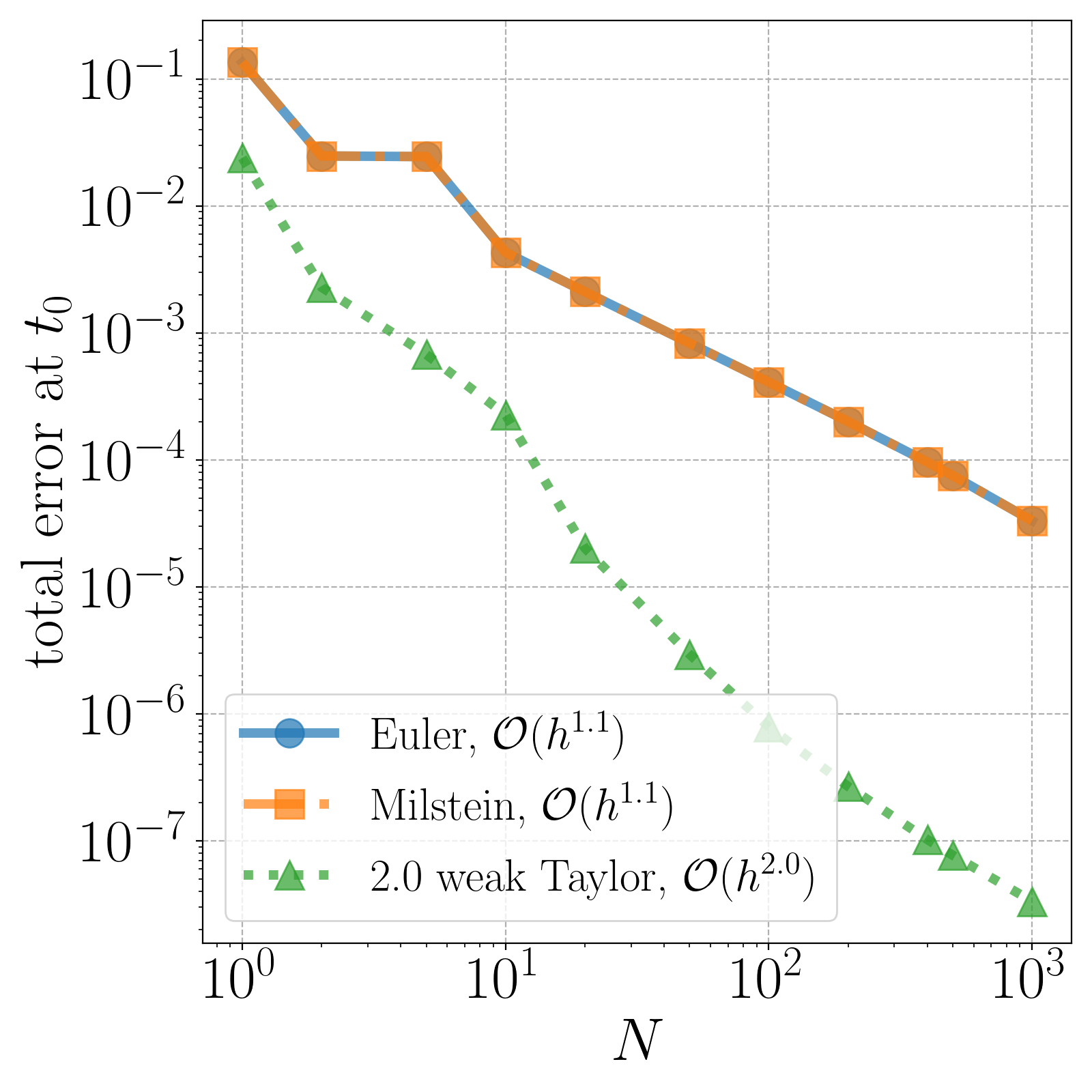}\includegraphics[width=\sizetwobyone\linewidth]{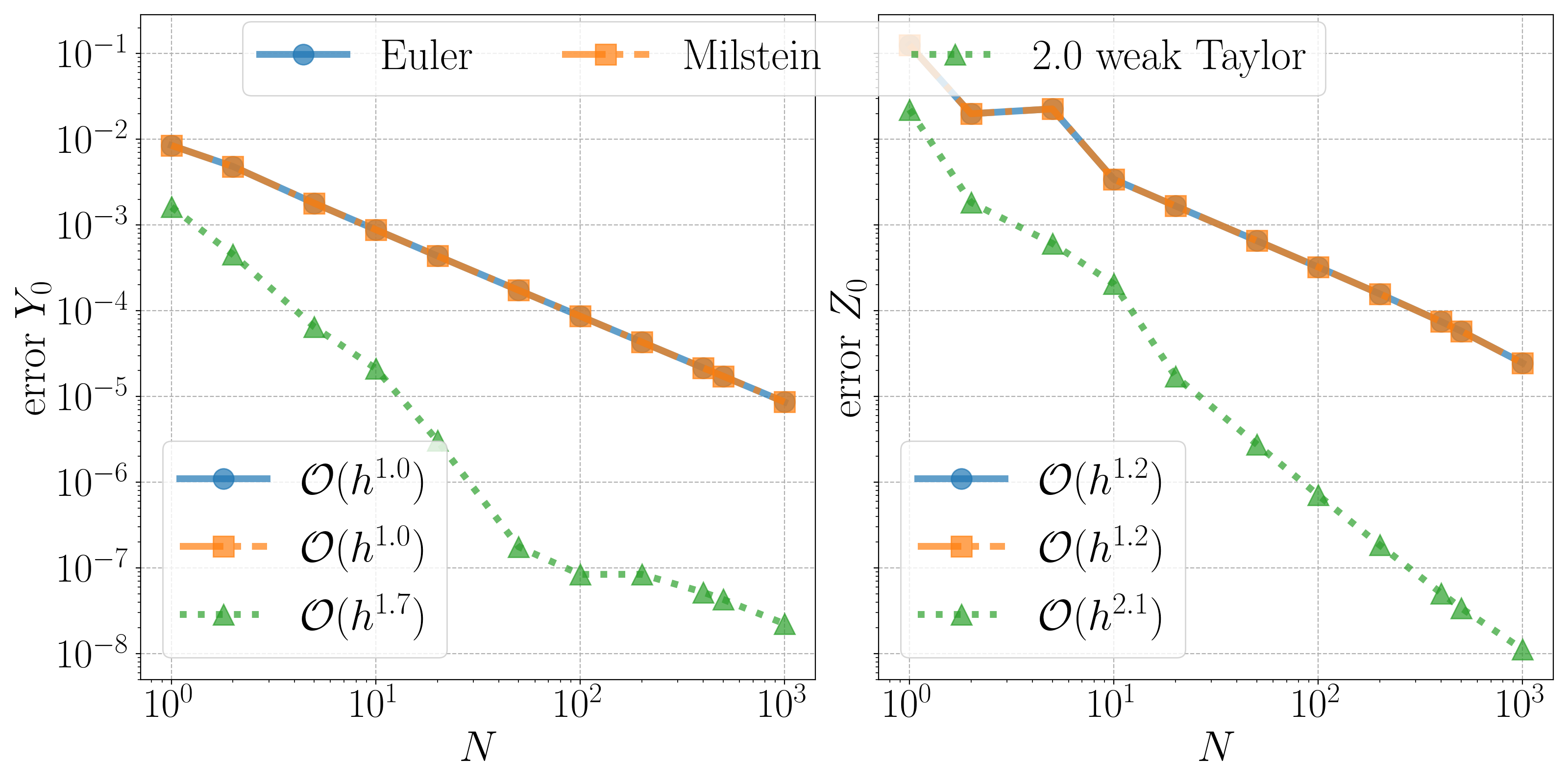}
        \caption{weak convergence at $t_0$}
        \label{fig:example3:weak}
        \end{subfigure}
        
        \caption{Example 3 in \eqref{eq:example3}. ($\theta_1=\theta_2=\theta_3=1/2$, $\theta_4=-1/2$, $K=2^{10}$.)}
        \label{fig:example3}
    \end{figure}
    
    \section{Conclusion}
    In this article, we extended the BCOS method of \cite{ruijter_fourier_2015, ruijter_numerical_2016, huijskens_efficient_2016} to second-order Taylor schemes approximating the Markov transition between two time steps in the fully-coupled FBSDE setting. We presented an algorithmic framework that unifies second-order Taylor schemes for fully-coupled equations, including the Euler-, Milstein- and simplified order 2.0 weak Taylor approximations for the forward SDE; and the generalized theta-scheme of \cite{zhao_generalized_2012} for the BSDE. Building on the closed-form expression for the characteristic function of the corresponding Markov transitions in lemma \ref{lemma:chf}, we extended the coupled BCOS method and gave an implementable, higher-order numerical method for the fully-coupled FBSDEs in algorithm \ref{algorithm}. We demonstrated the robustness and accuracy of our algorithm on a wide range of equations, spanning from the decoupled to the fully-coupled case, and found that the hereby proposed second-order Taylor methods bring an improved first-order strong convergence to the total approximation error of the FBSDE compared to the Euler scheme in \cite{huijskens_efficient_2016}, if the theta parameters of the backward discretization are chosen accordingly. Additionally, we found that the 2.0 weak Taylor discretization further improves the convergence of the approximation errors at $t_0$ to second-order, which is crucial in applications such as stochastic optimal control. On top of the improved accuracy, the methods also proved to be robust with respect to the number of terms in the finitely truncated Fourier expansion.

    \paragraph{Acknowledgments} B.N. acknowledges financial support from the Peter Paul Peterich Foundation via the TU Delft University Fund.
	
	\printbibliography[heading=bibintoc, title={References}]
	
\end{document}